\documentclass{amsart}
\usepackage{color}
\usepackage{graphicx}
\usepackage{hyperref}  
\usepackage{mathptmx}
\usepackage{amsmath}
\usepackage{amssymb}
 
\usepackage{pdfpages}

\renewcommand{\epsilon}{\varepsilon}

\newcommand{\N}{{\mathbb N}}
\newcommand{\Z}{{\mathbb Z}}

\newcommand{\R}{{\mathbb{R}}}
\newcommand{\C}{{\mathbb C}}

\newcommand{\Vol}{\operatorname{Vol}}
\newcommand{\vol}{\operatorname{vol}}

\newcommand{\Cusp}{\operatorname{Cusp}}

\newcommand{\SL}{\operatorname{SL}(2,\R)}
\newcommand{\hol}{\operatorname{hol}}

\newcommand{\cyl}{\operatorname{cyl}}
\newcommand{\conf}{\operatorname{conf}}
\newcommand{\area}{\operatorname{area}}
\newcommand{\meanarea}{\operatorname{mean~area}}
\newcommand{\MeanArea}{\operatorname{mean~area~conf}}

\newcommand{\thick}{\operatorname{thick}}

\newcommand{\even}{\operatorname{even}}
\newcommand{\odd}{\operatorname{odd}}
\newcommand{\hyp}{\operatorname{hyp}}

\newcommand{\I}{\operatorname{I}}
\newcommand{\B}{\operatorname{B}}

\newcommand{\dd}{\mathrm{d}}
\newcommand{\cC}{{\mathcal C}}

\newcommand{\cH}{{\mathcal H}}

\newcommand{\ds}{\displaystyle}

\newtheorem{theorem}{Theorem}
\newtheorem{corollary}[theorem]{Corollary}

\newtheorem{lemma}[theorem]{Lemma}
\newtheorem{proposition}[theorem]{Proposition}

\theoremstyle{definition}
\newtheorem*{remark}{Remark}

\newtheorem*{acknowledgement}{Acknowledgements}
{\bf}{\rm}
\newtheorem*{theocite}{Theorem}{\bf}{\it}

\begin{document}
\date{}
\title{Geometry of periodic regions on flat surfaces and associated Siegel--Veech constants}
\author[M.~Bauer]{Max Bauer}
\address[Max Bauer]{IRMAR,Universit\'e de Rennes 1, 35042 Rennes, France}
\email{maximilian.bauer@univ-rennes1.fr}
\author[E.~Goujard]{Elise Goujard}
\address[Elise Goujard]{IRMAR,Universit\'e de Rennes 1, 35042 Rennes, France}
\email{elise.goujard@univ-rennes1.fr}
\thanks{Both authors are partially supported by ANR ``GeoDyM''}
\maketitle


\begin{abstract}
An Abelian differential gives rise to a flat structure (translation surface)
on the underlying Riemann surface. 
In some directions the directional flow on the flat surface may contain a periodic region that is made up of maximal cylinders filled by parallel geodesics of the same length. 
The growth rate of the number of such regions counted with weights, as a function of the length, is quadratic with a coefficient, called Siegel--Veech constant, that is shared by almost all translation surfaces in the ambient stratum. 

We evaluate various Siegel--Veech constants associated to the geometry of configurations of periodic cylinders and their area, and study extremal properties of such configurations in a fixed stratum 
and in all strata of a fixed genus.
\end{abstract}

\section{Introduction}

\subsection{Statement of some known results}\label{sec:known-results}

Suppose that $M_{g,m}$ is a closed connected oriented surface $M_{g}$ of genus $g$ with a set of $m$ labelled marked points $\Sigma=\{P_1,\dots,P_m\}$. 

By a {\em translation surface} $S$ we mean a flat Riemannian metric and a parallel vector field on $M_{g,m}$. The metric has cone type singularities at all of the points $P_i$ of $\Sigma$ where the total angle is of the form $2\pi(d_i+1)$ for some integer $d_i\geq 1$. 

Each geodesic on a translation surface  moves in a constant direction, so geodesics do not have self intersections and a regular geodesic that connects a non-singular point to itself comes back with the same angle, so it is a periodic geodesic. A periodic geodesic is always part of a maximal connected periodic region: a maximal cylinder of parallel periodic geodesics of the same length. 
We refer to such a maximal cylinder as a {\em periodic cylinder} or, for short, a {\em cylinder}, and we say that the common  length  of the periodic geodesics that make up the cylinder is  the {\em width} of the cylinder. The number $N_{cyl}(S, \cC, L)$ of cylinders of width less than $L$ grows like $c\pi L^2$:
a first fundamental result of Masur \cite{Ma2,Ma3} states that  there exist two positive constants $c_1$ and $c_2$ such that 
$c_1\pi L^2\leq N_{cyl}(S, \cC, L)\leq c_2\pi L^2$. Using this fact, Eskin and Masur \cite{EM}
prove the deep theorem  that for almost every surface in a stratum there is an exact asymptotics 
$N_{cyl}(S, \cC, L)\sim c\pi L^2$. 
The constant $c$ is called a {\em Siegel--Veech constant}. The main tool to study Siegel--Veech constants is the method of Veech \cite{Ve3} that also showed the exact asymptotics in a more
general setting, but for a weaker form of convergence. 

The main object of this paper is the computation of various Siegel--Veech constants related to counting periodic regions in different ways.
Siegel--Veech constants are particularly interesting because they appear in various fields. 
In arithmetic, they give the asymptotics of the number of primitive points in certain lattices. 
In dynamics they are related to the sum of Lyapunov exponents for the Teichm\"uller geodesic flow in any invariant suborbifold of a stratum of Abelian differentials by the main formula of \cite{EKZ}. For a Teichm\"uller curve, they have an algebro-geometric interpretation involving degrees of certain line bundles given by a formula of Bouw-M\"oller \cite{BM}. 
Furthermore they are related to slopes of effective divisors and intersection theory in the  moduli space of curves: this aspect is studied by Chen (see \cite{C} for example) and Chen-M\"oller (see \cite{CM} for example).

By identifying $\R^2$ and $\C$, a translation surface inherits from $\C$ a complex structure on $M_g$ and a holomorphic  one form ({\em Abelian differential}) $\omega$.  A zero of $\omega$ of order $d$ corresponds to a conical singularity of angle $2\pi(d+1)$.
There is a one to one correspondence between translation surfaces endowed with a distinguished direction and Abelian differentials.
If we denote by $\alpha=(d_1,\ldots,d_m)$ the orders of the zeros of $\omega$  then we have $\sum d_i=2g-2$.

For a given  $\alpha=(d_1,\dots,d_m)$ such that $\sum d_i=2g-2$ and $d_i\geq 1$, for $i=1,\dots,m$, we consider the stratum  $\cH(\alpha)$ of the moduli space of Abelian differentials on $M_{g,m}$ that have zeros at the points of $\Sigma$ of orders $(d_1,\dots,d_m)$, or equivalently the moduli space of translation surfaces $S$ with singularities at  the points of $\Sigma$ of angles $(2\pi(d_1+1),\dots,2\pi(d_m+1))$. 
The dimension of $\cH(\alpha)$ is $\dim_\C\cH(\alpha)=2g+m-1$. 
The stratum $\cH(\alpha)$ may be non connected \cite{Ve2} but contains at most three connected components \cite{KZ}.
The stratum $\cH(\alpha)$ admits a volume element \cite{Ma1,Ve1} that induces a finite $SL(2, \R)$-invariant measure on the hyperspace $\cH_1(\alpha)$ of area one surfaces in $\cH(\alpha)$.
The volumes of  the connected components of the strata of Abelian differentials were effectively computed by A.~Eskin and A.~Okounkov \cite{EO}.

The boundary of a periodic cylinder contains singularities. Generically, each boundary component contains exactly one singularity so it is a {\em closed saddle connection}, i.e. a geodesic that joins a singularity to itself (and contains no other singularity).

Consider a translation surface $S$. For each positive real number $L>0$, denote by $N_{\cyl}(S,L)$ the number of periodic cylinders in $S$ of width at most  $L$ and by $N_{\area}(S,L)$ the total area of these cylinders.
It was shown in~\cite{EM} that 

\begin{theocite}[Eskin--Masur]
Let $\cH(\alpha)$ be a stratum, where $\alpha=(d_1,\dots,d_m)$ with $d_i\geq 1$, for $i=1,\dots,m$.
For every  connected component $K$ of  $\cH_1(\alpha)$, there exist constants  $c_{\cyl}(K)$ and $c_{\area}(K)$  such that for almost every translation surface $S$ in $K$ one has
\[\lim_{L\rightarrow \infty}\frac{N_{\cyl}(S,L)}{\pi L^2}=c_{\cyl}(K)\qquad
\lim_{L\rightarrow \infty}\frac{N_{\area}(S,L)}{\pi L^2}=c_{\area}(K).\]
\end{theocite}
The {\em Siegel--Veech constants} $c_{\cyl}(K)$ and $c_{\area}(K)$ only depend on $K$. An earlier version of this result (in a more general setting) where convergence is replaced by convergence in $L^1$ was proved in~\cite{Ve3}.

The results in~\cite{EM,Ve3}  assure the existence of quadratic asymptotics (Siegel--Veech constants) if one counts cylinders with weights (under certain conditions).  
The existence of all the Siegel--Veech constants we consider in this paper is justified by these results (see section~\ref{sec:notation}).
The function $N_{area}(S,L)$ for example counts cylinders with weight the area of the cylinder.

One can also only count  cylinders with sufficiently big area: for $x\in [0,1)$, denote by $N_{\cyl,A\ge x}(S,L)$ the number of cylinders in $S$ of width at most $L$ and of area at least $x$. We denote the corresponding Siegel--Veech constant by $c_{\cyl,A\ge x}(K)$.
It is shown in~\cite{Vo} that

\begin{theocite}[Vorobets~\cite{Vo}]\label{th:Vorobets} 
Let $\cH(\alpha)$ be a stratum, where $\alpha=(d_1,\dots,d_m)$ with $d_i\geq 1$, for $i=1,\dots,m$. Then for any connected component $K$ of  $\cH_1(\alpha)$
\begin{align*}
\mathrm{(a)}\quad &c_{\meanarea}(K)=\frac{c_{\area}(K)}{c_{\cyl}(K)}=
\frac{1}{2g+m-2}=\frac{1}{\dim_{\C} \cH(\alpha)-1}\\
\mathrm{(b)}\quad &\frac{c_{\cyl,A\ge x}(K)}{c_{\cyl}(K)}=(1-x)^{2g+m-3}=(1-x)^{\dim_{\C} \cH(\alpha)-2}.
\end{align*}
\end{theocite}

The ratio $c_{\meanarea}(K)$ of part a) of the theorem can be interpreted as the (asymptotic) mean area of a cylinder on a generic surface in $K$ in the following sense: for almost any surface  $M$ in $K$ one has $\displaystyle \frac{c_{\area}(K)}{c_{\cyl}(K)}=\lim_{L\rightarrow\infty}\frac{N_{\area}(M,L)}{N_{\cyl}(M,L)}$.

Part b) of the theorem is an answer to a question of Veech in~\cite{Ve3} where the author asks if there is a simple formula for  $\cfrac{c_{\cyl,A\ge x}(K)}{c_{\cyl,A\ge 0}(K)}$ (which is the same as 
$\cfrac{c_{\cyl,A\ge x}(K)}{c_{cyl}(K)}$).

As a by-product of our results using the methods from~\cite{EMZ} we get an alternative proof of the theorem of Vorobets by evaluating an explicit integral that is a simplified version of the integral 
used in~\cite{EMZ}. 
 
The methods from~\cite{EMZ} allow for the computation of various other Siegel--Veech constants associated to counting cylinders for more specific data that we allude to next.
 
It might happen that  the geodesic flow in a given direction on $M$ contains several (maximal) periodic cylinders in that direction. 
Generically, this only happens if the boundary saddle connections of the cylinders are homologous. The boundary saddle connections might be part of a larger family of homologous saddle connections, where the extra saddle connections do not bound a cylinder. We refer to such a family as a {\em configuration of homologous saddle connections} or simply as a {\em configuration}. A topological representation of such configurations can be obtained by taking blocks of surfaces as in figures~\ref{fig:type I}, \ref{fig:type II}, and~\ref{fig:type III} (the surfaces are drawn as tori, but might have arbitrary genus), arranging them in a cyclic order and then identifying the boundary components.

Note that the fact that the saddle connections are homologous implies that they are all of the same length. Call this the {\em length} of the configuration. It also implies that a configuration persists under small deformations of the translation structure.
 
One says that two configurations of homologous saddle connections correspond to the same {\em topological type} if the saddle connections are based at the same singularities, if the complementary regions are of the same topological type and have the same number and type of singularities, e.t.c. In particular, they always have the same number of complementary cylinders.  The length of the configuration is the common width of the cylinders coming from the configuration.
 
For each connected component $K$ of a stratum there are only  a finite number of {\em admissible} topological types of configurations, i.e. topological types of configurations that  are realized on  at least one surface in $K$. In fact almost all surfaces share the same admissible topological types. We will consider only configurations of this special type. Their distinguishing feature is that they persist under any small deformation. For related counting problems in the case of general periodic components (so not configurations), see \cite{N,L}.
 
For $S$ in $K$, denote by $N_{\conf}(S,\cC,L)$ the number of configurations of homologous saddle connections on $S$ of type $\cC$ and whose length is at most $L$.

\begin{theocite}[Eskin--Masur--Zorich, \cite{EMZ}]
For a given connected component $K$ of a stratum $\cH_1(\alpha)$ and an admissible topological type $\cC$ of configurations there exists a Siegel--Veech constant $c_{\conf}(K,\cC)$ such that for almost any surface $S$ in $K$ one has
\[\lim_{L\rightarrow \infty}\frac{N_{\conf}(S,\cC,L)}{\pi L^2}=c_{\conf}(K,\cC).\]
\end{theocite}
The authors of \cite{EMZ} give a method to compute these Siegel--Veech constants.

\subsection{Statement of results}\label{sec:results}
We denote by $N_{\cyl}(S,\cC,L)$ the number of cylinders 
of length less than $L$ coming  from a configuration of type $\cC$. For
a real number $p\ge 0$ we denote by
$N_{\area^p}(S,\cC,L)$ the sum of the $p$-th power of the area of each of these cylinders.
We denote the corresponding Siegel--Veech constants by 
$c_{\cyl}(K,\cC)$, resp.  $c_{\area^p}(K,\cC)$.
For $p=1$ we write $c_{\area}(K,\cC)$.
Note that if $\cC$ comes with $q$ cylinders then $c_{\cyl}(K,\cC)=qc_{\conf}(K,\cC)$.

It follows from the general result of~\cite{EM} that there is a Siegel--Veech constant 
$c_{\area^p}(K,\cC)$ such that for almost any surface $S$ in $K$ one has
\[\lim_{L\rightarrow \infty}\frac{N_{\area^p}(S,\cC,L)}{\pi L^2}=c_{\area^p}(K,\cC).\]

The methods from~\cite{EMZ}  can be applied to compute $c_{\area^p}(K,\cC)$ in a 
way similar to the computation of $c_{\conf}(K,\cC)$. The expression for 
$c_{\area^p}(K,\cC)$ contains a  constant $M$ that depends  
only on combinatorial data such as the dimension of the stratum, the order 
of the singularities and the possible symmetries. It is given by an explicit formula 
in \S~13.3. of \cite{EMZ}. It also contains the
 ``principal boundary stratum'' $\cH_1(\alpha')$ determined by $K$ and $\cC$
  (the stratum of possibly disconnected surfaces we get by contracting the closed 
 saddle connections of the configuration). The constant $n$ always denotes 
 the complex dimension of $\cH_1(\alpha')$. If  $\cC$ comes with $q$ cylinders 
 then we have $n=\dim_\C \cH(\alpha)-q-1$. 

With this notation we show in section~\ref{sec:mean-area},
\begin{theorem}\label{th:area-p} Given a real number $p\geq 0$.
Let $\cH(\alpha)$ be a stratum, where $\alpha=(d_1,\dots,d_m)$ and $d_i\geq 1$, for $i=1,\dots,m$.
Let $K$ be a connected component of  $\cH_1(\alpha)$ and $\cC$ an admissible topological type of configuration containing $q\geq 1$ cylinders. Then
\begin{align*}
c_{\area^p}(K,\cC)&=M \cdot\frac{\Vol(\cH_1(\alpha'))}{\Vol(K)}\cdot 
   \frac{(n-1)!}{(p+1)\cdot (p+2) \cdots (p+q+n-1)}\cdot q\\ 
c_{\cyl}(K,\cC)
      &=M \cdot\frac{\Vol(\cH_1(\alpha'))}{\Vol(K)}\cdot \frac{(n-1)!}{(n+q-1)!}\cdot q\\
c_{area}(K,\cC)&=
M\cdot \frac{\Vol(\cH_1(\alpha'))}{\Vol(\cH_1(\alpha))}\cdot \frac{(n-1)!}{(n+q)!}\cdot q
\end{align*}
where $n=\dim_\C \cH(\alpha)-q-1$. $M$
denotes the combinatorial constant given in \S~13.3. of \cite{EMZ} and $\cH_1(\alpha')$ denotes the principal boundary stratum.
\end{theorem}

Note that $c_{\area}(K,\cC)=c_{\area^1}(K,\cC)$ and $c_{\cyl}(K,\cC)=c_{\area^0}(K,\cC)$, so the second and third equation of the previous theorem follow from the first.

\begin{remark} 
\begin{enumerate}
\item Evaluation of $c_{area^p}$ is motivated by the question of M.~M\"oller related to the study of quasimodular properties of the related counting function. 
\item We recall that the saddle connections in a configuration 
 of given type $\cC$ can be named. Choose and fix one of the 
 saddle connections that bounds a cylinder. In the proof of
theorem~\ref{th:area-p} we show that if we only consider the  area of this cylinder
 then we get the same formulas for $c_{\area^p}(K,\cC)$ and $c_{\area}(K,\cC)$  except that the factor $q$ is missing.
 \end{enumerate}
 \end{remark}

We get as a corollary:

\begin{corollary}\label{cor:mean-area-p}
Given a real number $p\ge 0$ and let $\cH(\alpha)$ be a stratum of Abelian differentials on a surface $M_{g,m}$, 
where $\alpha=(d_1,\dots,d_m)$ with $d_i\geq 1$, for $i=1,\dots,m$.
Then for any connected component 
$K$ of  $\cH_1(\alpha)$ and any admissible  type $\cC$ of configuration containing at least one cylinder,
\[
c_{\meanarea^p}(K,\cC)=
\cfrac{c_{\area^p}(K,\cC)}{c_{\cyl}(K,\cC)}=
\cfrac{(d-2)!}{(p+1)\cdot (p+2)\cdots (p+d-2)},
\]
where $d=\dim_{\C} \cH(\alpha)=2g+m-1$.
\end{corollary}

\begin{remark} 
(a) The quotient of the corollary can be interpreted as the 
 (asymptotic)
mean area  of a cylinder coming
from a configuration of type $\cC$, where the area is counted with a power $p$.

(b) For  a natural number  $p\geq 1$ we obtain
\[ c_{\meanarea^p}(K,\cC)=
 \frac{1}{{p+d-2\choose p}}.
  \]

\end{remark}

Define $N_{\area^p}(S,L)$ in the same way as we defined $N_{\area}(S,L)$ in 
section~\ref{sec:known-results}, except that  the area of each cylinders is counted with a
power $p$. If $c_{\area^p}(K)$ denotes the corresponding Siegel-Veech constant, then
we have
\[ c_{\cyl}(K)=\sum_{\cC} c_{\cyl}(K,\cC)\quad\text{and}\quad 
   c_{\area^p}(K)=\sum_{\cC} c_{\area^p}(K,\cC),\]
where the sum is taken over all admissible topological types of configurations for $K$
with at least one cylinder. This implies

\begin{corollary}\label{cor:Vorobets-p} 
Given a real number $p\ge 0$ and let $\cH(\alpha)$ be a stratum, 
where $\alpha=(d_1,\dots,d_m)$ with $d_i\geq 1$, for $i=1,\dots,m$. 
Then for any connected component $K$ of  $\cH_1(\alpha)$
\[c_{\meanarea^p}(K)=\frac{c_{\area^p}(K)}{c_{\cyl}(K)}=
\cfrac{(d-2)!}{(p+1)\cdot (p+2)\cdots (p+d-2)},
\]
where $d=\dim_{\C} \cH(\alpha)=2g+m-1$.
\end{corollary}

For $p=1$, corollary~\ref{cor:Vorobets-p} becomes part (a)  of the theorem of 
Vorobets as stated in section~\ref{sec:known-results}.
Corollary~\ref{cor:mean-area-p}
gives more detailed
information as corollary~\ref{cor:Vorobets-p}.
For example, 
the mean area $c_{\meanarea^p}(K,\cC)$  of a cylinder coming from a 
configuration of type $\cC$
does not depend on the number of cylinders. This means in particular that the mean area of a cylinder 
is the same if the cylinder makes up the whole periodic region of a configuration or
if the periodic region is made up of several cylinders. 

As a variation of the above, we  denote by $N_{\area^p,\conf}(S,\cC,L)$ the 
$p$-th power of the
total area of the periodic region (union of the cylinders) on $S$ coming from a 
configuration of topological
type $\cC$ whose length is at most $L$. The corresponding Siegel-Veech constant
is denoted by $c_{\area^p,\conf}(K,\cC)$. We show in section~\ref{sec:mean-area-conf}

\begin{theorem}\label{th:area-p-conf}
Given a real number $p\geq 0$ and let $\cH(\alpha)$ be a stratum, where $\alpha=(d_1,\dots,d_m)$ and $d_i\geq 1$, for $i=1,\dots,m$.
Let $K$ be a connected component of  $\cH_1(\alpha)$ and $\cC$ an admissible topological type of configuration containing $q\geq 1$ cylinders. Then
\begin{align*}
\mathrm{(a)}\quad c_{\area^p,\conf}(K,\cC)&=
 M \cdot\frac{\Vol(\cH_1(\alpha'))}{\Vol(K)}\cdot \frac{(n-1)!}{(q-1)!}\frac{1}{(p+q)\cdots (p+q+n-1)}\\ 
\mathrm{(b)}\quad\cfrac{c_{\area^p,\conf}(K,\cC)}{c_{\conf}(K,\cC)}&=
\frac{q(q+1)\cdots (q+n-1)}{(p+q)(p+q+1)\cdots (p+q+n-1)}.
\end{align*}
 $M$ denotes the combinatorial constant given in \S~13.3. of \cite{EMZ} and 
 $\cH_1(\alpha')$ denotes the principal boundary stratum.
 \end{theorem}

\begin{remark}
(a)
 For a natural number $p\geq 1$ we obtain
\[
\cfrac{c_{\area^p,\conf}(K,\cC)}{c_{\conf}(K,\cC)}=
\frac{q\cdot (q+1)\cdots (q+p-1)}{(d-1)\cdot d\cdots(d+p-2)},\]
where $d=\dim_\C \cH(\alpha)-1$.

(b) The quotient in part (b) of the preceeding theorem can be 
interpreted as the asymptotic mean area of the periodic part (taking the $p$-th power
of the area). For $p=1$ we get $qc_{\meanarea}(K,\cC)$, which is consistent,
as  $c_{\meanarea}(K,\cC)$
is the mean area of a cylinder.

(c)  For $q=1$ we have $c_{\area^p,conf}(K,\cC)=c_{\area^p}(K,\cC)$.
\end{remark}
 
 We next count configurations with cylinders of large area.
We recall that the saddle connections in a configuration of given type $\cC$ can be named. Choose 
and fix one of the saddle connections that bound a cylinder and call this the first cylinder. 
Given $x\in [0,1)$. Denote by $N_{\conf,A_1\geq x}(S,\cC,L)$ the number of configurations of type  $\cC$ of length at most  $L$  and such that the area of the first cylinder is at least $x$ (of the area one surface $S$). We denote by $c_{\conf,A_1\geq x}(K,\cC)$ the corresponding Siegel--Veech constant. Note that we have $c_{\conf,A_1\geq 0}(K,\cC)=c_{\conf}(K,\cC)$.
We show in section~\ref{sec:onecylindersp}

\begin{theorem}\label{th:Vorobets:constant}
Given $x\in [0,1)$. Let $\cH(\alpha)$ be a stratum of Abelian differentials on a surface $M_{g,m}$, where $\alpha=(d_1,\dots,d_m)$ with $d_i\geq 1$, for $i=1,\dots,m$. Then for any connected component $K$ of  $\cH_1(\alpha)$ and any admissible  topological type $\cC$ of configuration for $K$ containing at least one cylinder,
\[\frac{c_{\conf,A_1\geq x}(K,\cC)}{c_{\conf}(K,\cC)}=(1-x)^{2g+m-3}=(1-x)^{\dim_\C \cH(\alpha)-2}.\] 
 \end{theorem}
 
 Summing over all configurations we get as a corollary part (b) of the theorem of Vorobets
 but our result contains more detailed information.
 
We next count configurations with periodic regions of large area. Let $\cC$ be a topological type of configuration that comes with $q\geq 1$ cylinders.
 For $x\in [0,1)$, denote by $N_{\conf,A\geq x}(S,\cC,L)$ the number of configurations  on $S$ of type 
 $\cC$ of length at most $L$ and such that the total area of the $q$ cylinders is at least $x$ (of the area one surface $S$).  We denote the corresponding Siegel--Veech constant by $c_{\conf,A\geq x}(K,\cC)$.
 
 The  {\em incomplete Beta function} $B(t;a,b)$ and the  \emph{regularized incomplete Beta function} $\I(t;a,b)$ are defined  for $t\in [0,1]$ by
\[B(t;a,b)=\int_0^tu^{a-1}(1-u)^{b-1} \dd u,\qquad
 \I(t;a,b)=\frac{B(t;a,b)}{B(1;a,b)}=\frac{B(t;a,b)}{B(a,b)}.\]
For more details about the incomplete Beta function see  section~\ref{sec:toolbox}: Figure~\ref{fig:graphesbis} represents the density function  $\frac{d}{dt}\I(t;a,b)$ for various values for $a$ and $b$.

With this notation we show in section~\ref{sec:allcylindersp}:

\begin{theorem}\label{th:big-area}
Given $x\in [0,1)$, and let $\cH(\alpha)$ be a stratum, where $\alpha=(d_1,\dots,d_m)$ with $d_i\geq 1$, for $i=1,\dots,m$. Suppose that $K$ is a connected component of  $\cH_1(\alpha)$ and that $\cC$ is an admissible topological type  of configurations for $K$ containing exactly $q$ cylinders.
\[
\frac{c_{\conf,A\geq x}(K,\cC)}{c_{\conf}(K,\cC)}=\I(1-x;n,q)=
   (1-x)^{n}\sum_{k=0}^{q-1}\binom{n-1+k}{k}x^k,\]
where $n=\dim_\C \cH(\alpha)-q-1=\dim_\C \cH(\alpha')$.
\end{theorem}

See figure~\ref{fig:graphes} for the graph of $\I(1-x;n,q)$ for various values for $n$ and $q$.

\begin{figure}[!htb]
\begin{center}\includegraphics[scale=0.5]{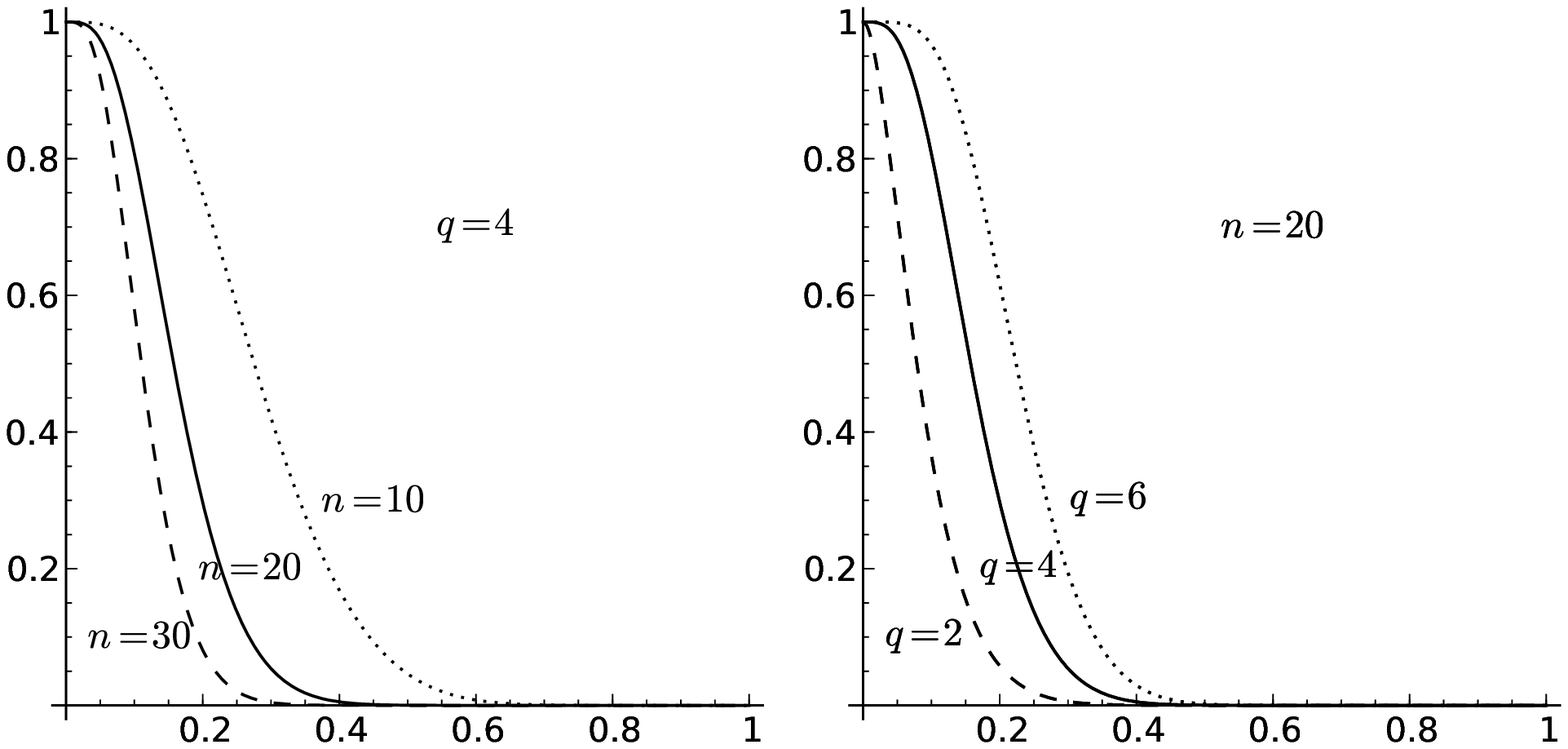}\end{center}
\caption{\label{fig:graphes} Graphs of the function 
$f(x)=\cfrac{c_{\conf,A>x}(K,\cC)}{c_{\conf}(K,\cC)}$}
\end{figure}

\begin{remark}
\begin{enumerate}
\item The fraction $\cfrac{c_{\conf,A\ge x}(K,\cC)}{c_{\conf}(K,\cC)}$ can be interpreted as the mean proportion of configurations of type $\cC$ whose periodic complementary region is big, that is, the total area of the cylinders is at least $x$ of the area of the surface.
 \item $\I(1;n,q)=1$ and \[\lim\limits_{x\to 1}\frac{I(0;n,q)}{(1-x)^n}=
 \sum_{\l=0}^{q-1}\binom{n+l-1}{l}=\binom{n+q-1}{n},\] so
\[\frac{c_{\conf,A\ge x}(K,\cC)}{c_{\conf}(K, \mathcal C)}\sim (1-x)^{n}\binom{n+q-1}{n}\mbox{ as }x\to 1.\]
In this form we can compare the result with the previous one for one cylinder, given in Theorem \ref{th:Vorobets:constant}.
\end{enumerate}
\end{remark}

In section~\ref{sec:correlation} we consider the problem of correlation between the area of two cylinders. Let $\cC$ be an admissible  configuration for a connected component $K$ that comes with at least two cylinders.
Choose (and fix) two cylinders and let $x,x_1\in [0,1)$. We denote by $N_{A_2\geq x, A_1\geq x_1}(S,\cC,L)$ the number of configurations of width length at least $L$ such that the area $A_1$ of the first cylinder is at least $x_1$ and such that the area $A_2$ of the second cylinder is at least $x$ of the remaining surface, i.e. it is at least $x(1-A_1)$. We denote by $c_{A_2\geq x, A_1\geq x_1}(K,\cC)$ the corresponding Siegel--Veech constant. To simplify notation we will write $c_{A_1\geq x_1}(K,\cC)$ instead of $c_{\conf,A_1\geq x_1}(K,\cC)$.
 We show in section~\ref{sec:correlation},
 
\begin{theorem}\label{th:correlation}
For any connected component $K$ of a stratum $\cH_1(\alpha)$, where $\alpha=(d_1,\dots,d_m)$  with $d_i\geq 1$, for $i=1,\dots,m$, and any admissible topological type $\cC$ of configurations  containing at least two cylinders,
 \[\frac{c_{A_2\geq x, A_1\geq x_1}(K,\cC)}{c_{A_1\geq x_1}(K,\cC)}=(1-x)^{\dim_\C \cH(\alpha)-3}.
 \]
 The result does not depend on $x_1$.
 \end{theorem}

Morally, we compute the asymptotic probability that, among configurations whose first cylinder has area $x_1$, we have a second cylinder with area at least $x(1-x_1)$. Comparing Theorems~\ref{th:correlation} and~\ref{th:Vorobets:constant} we see that this is the probability that, among all configurations, the area of the second cylinder is at least $x$, except that the parameter space has one fewer dimension.
So, in some sense, except for the fact that the area of the first cylinder gives a restriction on the range for the area of the second cylinder, the area of the second cylinder is independent of the area of the first cylinder.

In the results presented above we studied individual configurations. In the remaining part of the paper we study extremal properties of configurations among all configurations in a given stratum or even among all strata for a  fixed genus.
 
In section~\ref{sec:maximal-abelian} we address the question of finding topological types of configurations $\cC$ (admissible for some connected component $K$) that maximizes
\[c_{\MeanArea}(K,\cC)=
\frac{c_{\area}(K,\cC)}{c_{\conf}(K,\cC)}.\]
The constant $c_{\MeanArea}(K,\cC)$ can be interpreted as the asymptotic mean area of the periodic part (union of the cylinders) of the complementary region of a configuration of topological type $\cC$.

Each stratum has at most three connected components that are classified by the invariants ``hyperellipticity'', and ``parity of spin structure''.
(We recall in section~\ref{sec:simple-abelian}  the classification of connected components from~\cite{KZ}.)

The quantity $ c_{\MeanArea}(K,\cC) $ varies considerably among strata. For the connected stratum $\cH(1, 1,\dots, 1)$, the maximal value of $ c_{\MeanArea}(K,\cC) $ over all the configurations is $\frac{1}{4}$. For the connected component $\cH^{hyp}(g-1, g-1)$ it is equal to $\frac{1}{2g}$. The following proposition gives an uniform bound on the ratio $ c_{\MeanArea}(K,\cC) $.
\begin{theorem}\label{prop:qmax:dim:ab} 
Let $K$ be any connected component of a stratum $\cH(\alpha)$ and $\cC$ be any admissible topological type of configuration for $K$, then the asymptotic mean area of the periodic complementary regions satisfies
\[c_{\MeanArea}(K,\cC)\leq \frac{1}{3}.\]
The maximum is attained for any genus $g\ge 2$: for each $g\geq 2$ there is a topological type of configuration $\cC_g$ that is admissible for the component $\cH^{odd}(2,2,\dots,2)$ ($g-1$ zeros of order $2$) of Abelian differentials on a surface of genus $g$ such that the corresponding constant $c_{\MeanArea}$ is $\frac{1}{3}$.
\end{theorem}

To prove this result we need to determine the maximal number of cylinders that can come from any configuration which is admissible for a fixed stratum $\cH(\alpha)$.
We insist on the fact that we compute here the number of cylinders in rigid collections of saddle connections, which is different from the studies of Naveh \cite{N} and Lindsey \cite{L} where they count the number of parallel cylinders.

In section~\ref{sec:simple-abelian} we answer a question of A. Eskin and A. Wright: is it possible to find in each connected component of each stratum a topological type of configuration  whose complementary regions are tori with boundary and cylinders.

The answer depends on the connected component. 
We show (Proposition~\ref{prop:Alex}): for hyperelliptic components this is not possible; for the components with even spin structure when the genus is even this is not possible unless we allow one of the complementary regions to be a genus two surface; in all other cases this is possible.

\subsection{Notation}\label{sec:notation}
We introduce here most of the notation we need for the computation of Siegel--Veech constants. 
For survey material on Abelian differentials and translation surfaces see~\cite{GJ,Z,MT}. For the exact definition of configurations and related results  see~\cite{EMZ}.

Let $M_{g,n}$ denote a closed oriented surface $M_g$ of genus $g$ on which there are  $m$ marked points $\Sigma=\{P_1,\dots,P_m\}$. By $(R,\omega)$ we denote a Riemann surface structure $R$ on $M_g$ together with an  Abelian differential $\omega$.
If $\omega$ is not identically zero then $\Sigma$ is the set of zeros of $\omega$.  
We usually only write $\omega$ for $(R,\omega)$.
If we denote by $(d_1,\ldots,d_m)$ the orders of the zeros of $\omega$  then we have $\sum d_i=2g-2$. 

The form $\omega$ can be used to define an atlas of adapted coordinates on $R$. In these adapted coordinates the Abelian differential $\omega$ becomes  $dz$ in a neighborhood of any point of 
$R\setminus \Sigma$ and is $(d_i+1)w^{d_i} dw=d(w^{d_i+1})$ in a neighborhood of a point $P_i\in\Sigma$, where $d_i$ is the order of the zero $P_i$. Transition functions away from the zeros for these adapted coordinates  are translations. We refer to a surface  together with an atlas whose transition functions are translations as {\em translation surfaces}.
 
Using such an atlas, $R\setminus \Sigma$ inherits from the complex plane a flat (zero curvature) Riemannian metric. A zero of order $d_i$ of the Abelian differential (that is a regular point of the Riemann surface structure) corresponds to a conical singularity of the flat metric of total angle $2\pi(d_i+1)$. These points are  also called {\em saddles}.

The horizontal unit vector field on $\C$ pulls back by adapted coordinates to a horizontal unit vector field on $R\setminus \Sigma$. Away from the singularities, the leaves of the corresponding foliation  are geodesics with respect to the flat metric. In fact, for each $\theta\in [0,2\pi[$ we have a unit vector field and so a foliation in that direction that comes from the unit vector field in direction $\theta$ of the complex plane.

The converse construction is also possible: suppose that $S$ is a translation surface structure on $M_{g,m}$, i.e. an atlas on $M_{g,m}$ whose transition functions are translations.
A translation surface inherits from $\R^2$ a flat Riemannian metric and a parallel vector field on $M_{g,m}$.  We assume that $M_g$ is the metric completion of $M_{g,m}$. The points of $\Sigma$ that are not regular points for the metric, are cone type singularities  where the total angle is of the form $2\pi(d_i+1)$ for some $d_i\in\N$. 

By identifying $\R^2$ and $\C$, a translation surface inherits from $\C$ a complex (Riemann surface) structure on $M_g$ and a holomorphic one form ({\em abelian differential}) $\omega$. The zeros of $\omega$ are contained in $\Sigma$.

Each translation surfaces can be represented by a polygon in the plane whose edges come in pairs of parallel sides of the same length. Identifying each pair by a translation one obtains a translation surface. The vertices give rise to singularities (or regular points if the total angle is $2\pi$.)

We identify two translation surface structures if there is a bijection of the underlying topological
surface that is in local coordinates  a translation. We identify two Abelian differentials (for some complex structures) if they are biholomorphically equivalent. There is then a one to one correspondance between Abelian differentials and translation surfaces. We denote by $\cH$ the moduli space of Abelian differentials $\omega$ or equivalently of translation surfaces $S$. The moduli space $\cH$ is an algebraic variety.
 
 Given $\alpha=(d_1,\ldots,d_m)$ such that $2g-2=\sum_i d_i$. The set $\cH(\alpha)$ of Abelian differentials that share the same zero structure $\alpha$ is called a stratum. The stratum $\cH(\alpha)$ is an algebraic subvariety that admits a natural affine structure and a natural ``Lebesgue'' volume element induced by this affine structure \cite{Ma1,Ve1}. The dimension of $\cH(\alpha)$ is $\dim_\C\cH(\alpha)=2g+m-1$.

The area of the translation surface $S$ defined by the Abelian differential $\omega=\phi(z) dz$ is $ \int _S |\phi(z)|^2 \dd x \dd y$. We denote by $\cH_1(\alpha)$ the hyperspace of $\cH(\alpha)$ of area one surfaces. Masur~\cite{Ma1} and Veech~\cite{Ve1} showed that $\cH_1(\alpha)$ with the measure induced by the measure on $\cH(\alpha)$ is of finite volume.

 There is a natural $\SL$ action on $\cH(\alpha)$. An element $g\in\SL$ acts on local  coordinates by postcomposition of $g$. The measure on $\cH(\alpha)$ and $\cH_1(\alpha)$ is $\SL$ invariant. If we represent a translation surfaces  by a polygon in $\R^2$, then the action of $g$ is the usual action on $\R^2$.

Geodesic segments in the flat metric have constant angle with respect to the flat metric, so a geodesic segment $\gamma$ can be represented by a holonomy vector $\hol(\gamma)$ in $\R^2$. The angle and length of the vector is given by the direction and length of the geodesic segment. We will often use $\gamma$ both for the geodesic segment and for the holonomy vector.

As we said above, a closed geodesic is always contained in a maximal cylinder of closed geodesics and each boundary component of the cylinder is generically a closed saddle connection.
 As all of the geodesics in the cylinder are represented by the same vector we say that this vector represents the cylinder. Its length is  the {\em width} of the cylinder.

Suppose that on a translation surface we find in some direction a {\em configuration of homologous saddle connections}, meaning a maximal family of saddle connections, that are homologous relative to the singularities. All of the saddle connections are parallel and of the same length, so they share the same holonomy vector. Its length is the {\em length} of the configuration. 

A configuration defines the following data: the named singularities the saddle connections are based at; the topological type of the complementary regions; the knowledge of which saddle connection bounds which complementary region, so in particular the cyclic order of the complementary regions; the order of the singularities (if there are any)  in the interior of each complementary region; the singularity structure on the boundary (the type of boundary and angles at the boundary singularities to be described below). 

We say that two configurations are of the same {\em topological type} if they define the same data.
Almost all surfaces in a given connected component $K$ share the same topological types of configurations that can be realized on the surface. We talk about an {\em admissible topological type} for $K$.

Denote by $V(S,\cC)$ the (discrete) set of  holonomy vectors associated to configurations on $S$ of type $\cC$. We  are allowed to associate weights to the holonomy vectors. Denote by $B(L)\subset\R^2$ the disk of radius $L$ centered at the origin and by $N(S,\cC,L)$ the cardinality of $V(S,\cC)\cap B(L)$, where the elements of $V(S,\cC)\cap B(L)$ are counted with their weights. So if we write an element of $V(S,\cC)$ as $(v,w(v))$ where $w(v)$ is the weight of the holonomy vector $v$, then $N(S,\cC,L)=\sum_{v\in V(S,\cC)\cap B(L)}w(v)$.

By using appropriate weights on the holonomy vectors the counting function $N(S,\cC,L)$ becomes the counting functions $N_{\conf}(S,\cC,L)$, $N_{\cyl}(S,\cC,L)$, $N_{\area}(S,\cC,L)$ e.t.c. introduced in section~\ref{sec:results}.
If for example we count holonomy vectors associated to a configuration with weight one (resp. the number of cylinders, total area of the cylinders) then $N(S,\cC,L)$ equals $N_{\conf}(S,\cC,L)$ (resp. $N_{\cyl}(S,\cC,L)$, $N_{\area}(S,\cC,L)$).

Given a connected component $K$ of some stratum $\cH_1(\alpha)$ and an admissible topological type $\cC$ for $K$. It follows from~\cite{EM} that the set of (weighted) holonomy vectors $V(S,\cC)$ and the associated counting functions $N(S,\cC,L)$ we consider in this paper verify the following conditions :

\begin{enumerate}
\item[(A)] for every $g\in SL(2,\R)$, $V(gS,\cC)=gV(S,\cC)$.
\item[(B)] for every $S\in K$ there exists a constant $c(S)>0$ such that $N(S,L,\cC)\leq c(S) L^2$. The constant $c(S)$ can be chosen uniformly on compact sets of $K$.
\item[(C)] there exist constants $L>0$ and $\epsilon>0$ such that 
$N(S,L,\cC)$ is $L^{1+\epsilon}(K,\mu)$ as a function of $S$.
\end{enumerate}

In fact, the authors of~\cite{EM} show that the above conditions are verified for the set $V(S)$ of holonomy vectors of closed saddle connections on $S$ where the weight associated to a holonomy vector is the number of saddle connections on $S$ that share this vector.
The weights we use in this paper are ``invariant under $\SL$'', so the sets $V(S,\cC)$ we use verify condition (A). The number of configurations is bounded by the number of saddle connections, and as we use bounded weights, the counting functions $N(S,\cC,L)$ we consider also satisfy conditions (B) and (C). 

For $f\in C_0^{\infty}(\R^2)$ one defines the function $\hat f :K\rightarrow \R$ by
\[ \hat{f}(S)=\sum_{v\in V(S,\cC)}w(v) f(v).\]

The fact that the sets $V(S,\cC)$ and associated counting functions $N(S,\cC,L)$ we consider in this paper satisfy the conditions (A), (B), (C), implies, using~\cite{EM}, that:

\begin{theocite}[\cite{EM,Ve3}] Let $K$ be a connected component of some stratum $\cH_1(\alpha)$ of Abelian differentials and $\cC$ an admissible topological type of saddle connections for $K$. 
Then for any of the sets $V=V(S,\cC,L)$ and associated counting functions $N_V(S,\cC,L)$ we consider in this paper there is a constant $c_V(K,\cC)$ such that the following holds :
\begin{enumerate}
\item[(a)] For almost any  translation surface $S$ in $K$,
\[\lim_{L\rightarrow \infty}\frac{N_V(S,\cC,L)}{\pi L^2}=c_V(K,\cC).\]
\item[(b)] For any $f\in C_0^{\infty}(\R^2)$,
\[\frac{1}{\vol(K)}\int_K\hat{f}(S) \dd vol(S)=c_V(K,\cC)\int_{\R^2} f(x,y) \dd x \dd y.\]
\end{enumerate}
\end{theocite}

If the convergence in (a) is replaced by convergence in $L^1$ then this theorem follows from a more general theorem under similar hypotheses in~\cite{Ve3}. In that paper, (b) is proved for integrable functions $f$ of compact support. 
We will only use (b) for the characteristic function on a disk.

The paper~\cite{EMZ} explains how this last theorem can be used to compute the Siegel--Veech constant $c_{\conf}(K,\cC)$. The same method can be used to compute the other Siegel--Veech constants we consider in this paper. 

The strategy used in~\cite{EMZ} is as follows: 
if we apply (b) for the characteristic function $f_{\epsilon}$ of the disc $B(\epsilon)$ of radius $\epsilon$ centered at the origin then the integral of the right hand side becomes $\pi\epsilon^2$ and we have 
\[ \hat{f_{\epsilon}}(S)=\sum_{v\in V(S,\cC)\cap B(L)}w(v)=N(S,\cC,\epsilon).\]
So
\[  c_V(K,\cC)=\frac{1}{\vol(K)}\frac{1}{\pi\epsilon^2}\int_KN_V(S,\cC,\epsilon) \dd \vol(S).\]

Denote by $K(\epsilon,\cC)$ the subset of translation surfaces in $K$ that contain at least one configuration of type $\cC$ of length smaller than $\epsilon$. 
So $N_V(S,\cC,\epsilon)$ is zero outside $K(\epsilon,\cC)$.
Denote by $K^{\thick}(\epsilon,\cC)$ the subset of $K(\epsilon,\cC)$ of translation surfaces that contain exactly one configuration of type $\cC$ of length smaller than $\epsilon$ but contain no other closed saddle connection of length smaller than $\epsilon$. Using a result from~\cite{EM} the autors of~\cite{EMZ} show that
\[\vol(K(\epsilon,\cC))=\vol(K^{\thick}(\epsilon,\cC))+o(\epsilon^2).\] So
\begin{align}
  c_V(K,\cC)\quad=\quad&\frac{1}{\vol(K)}\frac{1}{\pi\epsilon^2}\int_{K(\epsilon,\cC)} N_V(S,\cC,\epsilon) 
  \dd \vol(S)\notag \\
  \quad=\quad&
  \lim_{\epsilon\to 0}\frac{1}{\vol(K)}\frac{1}{\pi\epsilon^2}\int_{K^{\thick}(\epsilon,\cC)} N_V(S,\cC,\epsilon) \dd \vol(S).\label{Siegel--Veech}
  \end{align}

Consider a translation surface $S$ in $K^{\thick}(\epsilon,\cC)$. On $S$ we have a configuration of closed saddle connections of topological type $\cC$ of length smaller than $\epsilon$ and no other short closed saddle connection. Cutting along the closed saddle connections we decompose
$S$ into several pieces. 
There will be some, say $p\geq 1$, surfaces $S_1,\dots,S_p$ with boundary and some, say $q\geq 0$, periodic cylinders $C_1,\dots,C_q$, all having the same width. 
The pieces are arranged in a cyclic order. The boundary of each $S_i$ is made up of two closed saddle connections. 

For each connected component $S_i$, by taking out the boundary and then taking the compactification we get a surface with either one boundary component, a {\em figure eight}, as in the right part of figure~\ref{fig:figureeight}, or we get a surface with two boundary components, a {\em pair of holes}, as in the right part of figure \ref{fig:twoholes}.

In the first case, the figure eight boundary describes two interior sectors of the surface of angles  $2\pi(a'+1)$ and $2\pi(a''+1)$, for some integers $a',a''\geq 0$. (Figure~\ref{fig:figureeight} illustrates the case $a'=1$ and $a''=0$.) By shrinking the figure eight boundary to a point we produce a singularity of order $a'+a''$ if $a'+a''\geq 1$, or a regular point if $a'+a''=0$. (See the left part of figure~\ref{fig:figureeight}.)

\begin{figure}[!htb]
\begin{center}\includegraphics[scale=1]{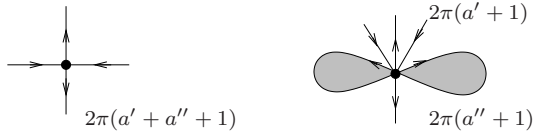}\end{center}
\caption{
\label{fig:figureeight}
Figure eight construction}
\end{figure}

In the second case, each of the boundary components comes with a boundary singularity of angles $\pi(2b'+3)$ and $\pi(2b''+3)$ for some integers $b',b''\geq 0$. (Figure~\ref{fig:twoholes} illustrates the case $b'=1$ and $b''=0$.)
By shrinking the two boundary components we produce singularities (or regular points) of orders $b'$ and $b''$. (See the left part of figure~\ref{fig:twoholes}.)

\begin{figure}[!htb]
\begin{center}\includegraphics[scale=1]{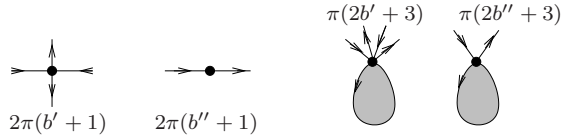}\end{center}
\caption{
\label{fig:twoholes}
Creating a pair of holes}
\end{figure}

The type of boundary, figure eight or pair of holes, and the associated angles is what we referred to above as ``singularity structure on the boundary''.

We get in this way closed surfaces $S_i'$ that  belong to some stratum $\cH(\alpha_i')$, for $i=1,\dots,p$. 
We write $\alpha'=\sqcup_{i=1}^p \alpha_i'$ and $\cH(\alpha')=\Pi_{i=1}^p\cH(\alpha_i')$.  We will say that the surface $S'$ whose connected components are $S_1',\dots,S_p'$ belongs to  $\cH(\alpha')$. We say that $S'$ belongs to the {\em principal boundary} of $\cH(\alpha)$ determined by $\cC$ and that $\cH(\alpha')$ is the corresponding {\em principal boundary stratum}. All this data (topological type of complementary regions, cyclic order, boundary stratum, e.t.c.) are the same for any configuration of the same topological type
$\cC$. In fact the  topological type is characterized by this data.

By shrinking a closed saddle connection, the singular point it is based at might become a regular point with total angle $2\pi$. 
In this case the regular point will be considered as a marked point of order $0$. So $\alpha_i'$ 
can contain one or two $0$.
 
 The procedure of shrinking saddle connections can be reversed. We summarize the description form~\cite{EMZ}  in case where $\cH(\alpha)$ has only one connected component $K$.
Start with a (maybe disconnected) surface $T'$ in $\cH(\alpha')$ and call the connected 
components $T_i'$. 
Choose some holonomy vector $\gamma$ in $B(\epsilon)$. 
There are two types of surgery. A {\em figure eight surgery} where we start with a singularity or a marked point of order $a\ge 0$, choose $a',a''\ge 0$ such that $a=a'+a''$ and then metrically create a figure eight boundary that consists of two saddle connections in direction $\gamma$ that are of length $|\gamma|$ and are based at the same singularity. There will be two sectors of angles $2\pi(a'+1)$ and $2\pi(a''+1)$  (see figure~\ref{fig:figureeight}).
A {\em pair of holes surgery} where we start with two points that are either  a singularity or a marked point of orders $b'\ge 0$ and $b''\ge 0$ and then create a boundary saddle connection at each singularity.
The pair of holes boundary has two boundary singularities of angles $\pi(2b'+3)$ and  $\pi(2b''+3)$ (see figure~\ref{fig:twoholes}).

By performing an appropriate surgery on each $T_i'$ we obtain surfaces $T_i$ that are   homeomorphic to the surfaces $S_i$ and have the same type of singularities in the interior and on the boundary. We then take $q$ cylinders $C_j$ with a marked point on each boundary. We finally combine the surfaces $T_i$ and the cylinders $C_j$ in the way prescribed by the topological type $\cC$ to produce a surface $T$ in $K^{\thick}(\epsilon,\cC)$. 
We do this by identifying pairs of boundary components by an isometry that identifies boundary singularities. The boundary components give rise to a configuration of closed saddle connections of topological type $\cC$. In fact each surface in $K^{\thick}(\epsilon,\cC)$ can be produced in this way.

The parameters used to produce surfaces in $K^{\thick}(\epsilon,\cC)$ are the following:
\begin{itemize}
\item a maybe disconnected surface in $\cH(\alpha')$.
\item a holonomy vector $\gamma$ in $B(\epsilon)$.
\item a combinatorial constant $M$ that only depends on the configuration. There is a $M:1$ correspondence between the surfaces in $K^{\thick}(\epsilon,\cC)$ and the surfaces in $\cH_1(\alpha')$. This is mainly due to the facts that at a zero of order $k$ there are $k+1$ 
sectors of angle $2\pi$ where we can produce a saddle connection in the direction of $\gamma$ and to 
possible symmetries of the surface in $K^{\thick}(\epsilon,\cC)$ (see \S~13.3. of \cite{EMZ}).
 \item the heights $h_i$ of the $q$ cylinders $C_i$ (the width is given by $|\gamma|$).
 \item for each cylinder $C_i$ a twist parameter $t_i \in [0,|\gamma|)$ that describes the relative position of the marked points on the two boundary components.
 \end{itemize}
 
\begin{remark}
If $\cH(\alpha)$ has more than one connected component than 
the correspondence between the thick part $K^{\thick}(\epsilon,\cC)$ 
of a connected component $K$ of $\cH_1(\alpha)$ and the principal boundary 
$\cH(\alpha')$ is slightly more complicated. For example, to construct a surface 
$T$ in $K^{\thick}(\epsilon,\cC)$, where $K$ is a hyperelliptic component, we must start with a surface $T'$ in the principal boundary such that all connected components $T_i'$ of $T'$ are hyperelliptic surfaces.
 So only the hyperelliptic components of some strata are in the principal boundary of  $K$. It might 
 even happen that only part of a connected component of a stratum is in the principal boundary. We still denote the principal boundary of $K^{\thick}(\epsilon,\cC)$ by $\cH(\alpha')$, although $\cH(\alpha')$ might be the union of (parts of) connected components of some strata.
 
We remark in passing that there are also some parameters used to describe the surgeries but this does not affect our computations. 
\end{remark}
If $q>1$, then to parametrize the $q$ tori we will always replace $h_q$ by $h=h_1+\dots+h_q$.
So the heights $(h_1,\dots,h_{q-1})$ are in the cone 
$\Delta^{q-1}(h)$ given by the conditions $h_i>0$, for $1\leq i\leq q-1$, and $h_1+\dots+h_{q-1}<h$.

For a given $\epsilon>0$ we define
\begin{multline*}
\cH^{\epsilon}(0^q)=
     \{(\gamma, h,h_1,\dots,h_{q-1},t_1,\dots,t_{q})\;|\;\\
      \gamma\in B(\epsilon),  (h_1,\dots,h_{q-1})\in \Delta^{q-1}(h),
           (t_1,\dots,t_{q})\in[0,|\gamma|]^q\}.\qquad
\end{multline*}
In what follows, $|\gamma|$ will always be small, 
but if no specific restriction on $|\gamma|$ is needed we will write $\cH(0^q)$.
We write $\dd\nu(T)$ for the measure 
\[
\dd\nu(T)=\dd\gamma\; \dd h\;\prod_{i=1}^{q-1} \dd h_i\,\prod_{i=1}^q\dd t_i.\]

We refer to the elements of $\cH^{\epsilon}(0^q)$ as tori as we 
obtain a torus with $q$ marked points by joining the $q$ cylinders. 
We write $\cH_1^{\epsilon}(0^q)$ for the subset of $\cH^{\epsilon}(0^q)$ of area $1$ tori, 
meaning that they satisfy the condition $h|\gamma|=1$.

We remark in passing that $\cH^{\epsilon}(0^q)$ can be interpreted as the $\epsilon$ -neighborhood of the ``cusp'' of the moduli space of flat tori with $q$ marked points  except that the marked points are already named by the way we parametrize them.
This is in contrast to the cusp of the usual moduli space $\cH^{\epsilon}(\underbrace{0,\dots,0}_{\textrm{$q$ times}})$ where the marked points can be arbitrarily named.

We denote by $\dd\nu(S)$ (resp. $\dd\nu(S')$) the measure on $\cH(\alpha)$ (resp. $\cH(\alpha')$).
It is shown in \cite{EMZ} that 
\[\dd\nu(S)=\dd\nu(S')\cdot \dd\nu(T).\]

Let  $S$ be a translation surface in $\cH_1(\alpha)$.  For $r$ a positive real number we denote by $rS\in\cH(\alpha)$ the surface we get by multiplying the flat metric on $S$ by $r$. Equivalently, if we represent $S$ by an Abelian differential $\omega$ with respect to some complex structure, then $rS$ corresponds to the Abelian differential $r\omega$ with respect to the same complex structure.
Note that we have in particular $\area(rS)=r^2\area(S)$.

If $X$ is a subset of $\cH_1(\alpha)$ then the cone $C(X)$ is defined to be \[C(X)=\{rS\,|\,0<r<1, S\in X\}\subset \cH(\alpha).\]
If we denote by $\dd\vol(S)$ the measure on  $\cH_1(\alpha)$ induced by the measure $\dd\nu(S)$ then we have
\[\dd\nu(S)=r^{\dim_\R \cH(\alpha)-1}\dd r\;\dd\vol(S).\]
So in particular one has 
\[ \vol(\cH_1(\alpha))=\dim_\R \cH(\alpha) \nu(\cH(\alpha)).\]
 
One has analogous statements for $\cH_1(\alpha')$ and $\cH_1(0^q)$ with their induced measures denoted by $\dd\vol(S')$ and $\dd\vol(T)$.

Note that there are $q+1$ complex parameters for $\cH^{\epsilon}(0^q)$ and that one has $\dim_\R\cH(\alpha)=\dim_\R\cH(\alpha')+\dim_\R\cH(0^q)$.
So by writing $n=\dim_{\C}\cH(\alpha')$ we have
\begin{align*}
\dim_{\R}\cH(\alpha')&=2n\\
\dim_{\R}\cH(0^q)&=2(q+1)\\
\dim_\R\cH(\alpha)&=2(n+q+1)
\end{align*}
We recall for completeness that if   $\alpha=(d_1,\dots,d_m)$ with $d_i\geq 1$ for $i=1,\dots,m$, then
\[\dim_\R\cH(\alpha)=2(2g+m-1).\]

\section{Siegel--Veech constants.}
 
 \subsection{General method}\label{sec:general-method}
We will describe in this section the method used to compute the Siegel--Veech constants introduced in section~\ref{sec:results}. 
Consider a stratum $\cH(\alpha)$ of Abelian differentials where $\alpha=(d_1,\dots,d_m)$ and $d_i\geq 1$ for $i=1,\dots,m$.  Suppose that $K$ is a connected component of  $\cH_1(\alpha)$ and that $\cC$ is an admissible topological type  of configurations for $K$ containing exactly $q\geq 1$ cylinders. We denote the corresponding principal boundary stratum by  $\cH(\alpha')$. Recall that $K^{\thick}(\epsilon,\cC)$ is the set of translation surfaces in $K$ that contain exactly one configuration of type $\cC$ of length smaller than $\epsilon$ but do not contain  another closed saddle connection of length smaller than $\epsilon$. 
To simplify notation we will write from now on $K^{\epsilon}$ for $K^{\thick}(\epsilon,\cC)$.

Suppose that $N(S,\cC,\epsilon)$ is one of the above mentioned counting functions that counts configurations on $S$ of length smaller than $\epsilon$ and of topological type $\C$. 
We want to evaluate $\int_{K^{\epsilon}} N(S,\cC,\epsilon) \dd\nu(S)$.

The decomposition of a surface $S\in \cH(\alpha)$ as $q$ tori with marked points on the boundary and a surface $S'$ in $\cH(\alpha')$ on which one performs surgeries ($M$ choices) gives the following parametrisation of the cone $C(K^{\epsilon})$: before performing the surgeries, a surface $S$ in $C(K^{\epsilon})$ is made up of $s S_1'$ and $tT_1$ for some scalars $s,t$ where $S_1'\in \cH_1(\alpha')$ and $T_1\in \cH_1(0^q)$. 
The conditions are the following :\\
 (i) the total area $s^2+t^2$ of $S$ satisfies $s^2+t^2\leq 1$; \\
The area $1$ surface $\frac{1}{\sqrt{s^2+t^2}}S$ is in $K^{\epsilon}$
so the waist curve of $\frac{t}{\sqrt{s^2+t^2}}T_1$ is smaller than $\epsilon$ which means
that the waist curve of $T_1$ is smaller than
$\epsilon\frac{\sqrt{s^2+t^2}}{t}$. So the second condition
is: \\
(ii) $T_1$ has to lie in $\cH_1^{\epsilon'}(0^q)$, where 
$\epsilon'=\epsilon\frac{\sqrt{s^2+t^2}}{t}$.

We have 
\begin{multline}\label{initial-integral}
\int_{C(K^{\epsilon})} N(S,\cC,\epsilon) \dd\nu(S)=\\
=M\vol(\cH_1(\alpha'))
\int\limits_0^1 s^{2n-1} \dd s
\int\limits_0^{\sqrt{1-s^2}} t^{2q+1}\dd t\int\limits_{\cH_1^{\epsilon'}(0^q)}N(S,\cC,\epsilon)\dd \vol (T)+ o(\epsilon^2),
\end{multline}
where $\epsilon'=\epsilon'(s,t)=\frac{\epsilon\cdot\sqrt{s^2+t^2}}{t}$. Note that if we construct $S\in C(K^{\epsilon})$ as described above then $S$ contains exactly one configuration of type $\cC$ and of width less than $\epsilon$, so $N(S,\cC,\epsilon)$ can be replaced by the weight with which we count configurations.

This integral is a simplification of the integral from~\cite{EMZ}, page 133. In~\cite{EMZ} the statement is about the volume of  $C(K^{\epsilon})$ (in our notaton) as the authors only count configurations with weight one, so $N(S,\cC,\epsilon)=1$ on $K^{\epsilon}$.

 \subsection{Mean area of a cylinder}\label{sec:mean-area}
  
Fix an admissible configuration $\cC$ in a connected component $K$ of a stratum 
$\cH_1(\alpha)$ that comes with $q\geq 1$ cylinders and let $p$ be a real number $p\ge 0$.

\begin{proof}[Proof of Theorem~\ref{th:area-p}]
 We choose and fix one of the named closed saddle connection of $\cC$ that bounds a cylinder and call this cylinder the first one. 
We denote by $N_{\area_1^p}(S,\cC,\epsilon)$ the number of
configurations on $S$ of type $\cC$ of length at most $\epsilon$,
counted with   weight the p-th power of  the area of the first cylinder.
We denote the corresponding Siegel-Veech constant by 
$c_{\area_1^p}(K,\cC)$.

As described above, we decompose a surface $S$ in $C(K^{\epsilon})$ as $sS_1'$ and $tT_1$.
We use  the usual parameters
  $\gamma,h,h_1,\dots,h_{q-1},t_1,\dots,t_q$ to parametrize 
$T_1\in \cH_1^{\epsilon}(0^q)$. We assume the notation chosen so that the first cylinder of $S$
corresponds to the first cylinder of $T_1$. So the
area of the first cylinder of $S$ is $t^2 h_1w$, where
 $w=|\gamma|$.
We use the following weight:
\[\left(\frac{t^2h_1w}{s^2+t^2}\right)^p.\]
This weight is invariant under scaling of $S$ and if the area
$s^2+t^2$ of $S$ equals $1$ then 
 it reduces to $(t^2h_1w)^p$ which is the $p$-th power of the area  of the first cylinder of $S$. 
A surface $S\in C(K^{\epsilon})$ contains exactly one configuration of type $\cC$
of width at most $\epsilon$, so 
we need to evaluate the integral of equation~\eqref{initial-integral} for 
\[N(S,\cC,\epsilon)=N_{\area_1^p}(S,\cC,\epsilon)=
\left(\frac{t^2}{s^2+t^2}\right)^p (h_1w)^p. \]

The first step to do this is to show that

\begin{lemma}\label{lem:cusp-p}
\begin{equation*}
\int_{\cH_1^{\epsilon}(0^q)}(h_1w)^p \dd\vol (T)=\frac{2\pi\epsilon^2}{(p+1)\cdot (p+2)\cdots (p+q-1)},
\end{equation*}
where  for $q=1$ the denominator is by convention equal to $1$.
\end{lemma}

\begin{proof} For $q=1$ we have $h_1w=hw=1$, so the computation is a simplified
version of the general case. Assume that $q>1$.
We will first integrate over the domain $C(\cH_1^{\epsilon}(0^q))$, still using the parameters
$\gamma,h,h_1,\dots,h_{q-1},t_1,\dots,t_q$ for 
$T\in C(\cH_1^{\epsilon}(0^q))$. 
To have a weight that is invariant under scaling of $T$ and that becomes
 $(h_1w)^p$ for an area one torus $T_1$, we use the weight
 \[\left(\frac{h_1w}{hw}\right)^p=\left(\frac{h_1}{h}\right)^p.\]
The domain of integration for $T\in C(\cH_1^{\epsilon}(0^q))$ is as follows:
The area $w h$ of $T$ satisfies $w h\leq 1$. 
We have $\frac{1}{\sqrt{wh}} T\in \cH_1^{\epsilon}(0^q)$, so the length $\frac{w}{\sqrt{wh}}$ of the waist curve of $\frac{1}{\sqrt{wh}} T$ is smaller than $\epsilon$, which gives $h\geq \frac{w}{\epsilon^2}$.
(See figure~\ref{fig:domain:of:integration:2}.)

\begin{figure}[!htb]
\begin{center}\includegraphics[scale=1]{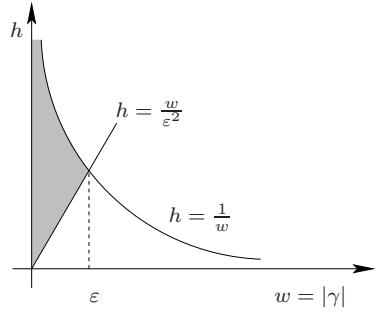}\end{center}
\caption{\label{fig:domain:of:integration:2}
Domain of integration for
$C(\cH^{\epsilon}_1(0))$}
\end{figure}

\begin{multline}\label{eq:cusp}
\int\limits_{C(\cH_1^{\epsilon}(0^q))}\left(\frac{h_1}{h}\right)^p\dd\vol (T)=
\int\limits_0^{2\pi} \dd\theta\int\limits_0^\epsilon w\; \dd w \int\limits_{w/\epsilon^2}^{1/w} \dd h
           \int_0^h\left( \frac{h_1}{h}\right)^p\dd h_1\cdot\\
           \cdot\int\limits_ {\Delta^{q-2}(h-h_1)} \dd h_2\dots \dd h_{q-1} 
        \int\limits_{ [0,w]^q}\dd t_1\dots \dd t_{q}  .
\end{multline}

The volume of the cone  $\Delta^{q-2}(h-h_1)$ is $\ds\frac{(h-h_1)^{q-2}}{(q-2)!}$ and 
the volume of the cube $[0,w]^q$ is $w^q$. Using the change of variables $u=h_1/h$,  
the right hand side of equation~\eqref{eq:cusp} then evaluates to
\[\frac{\pi\epsilon^2}{(q-2)!} \frac{\B(p+1,q-1)}{q+1}=
\frac{\pi\epsilon^2}{q+1}\frac{1}{(p+1)\cdots (p+q-1)},\]
where $\B(\cdot ,\cdot)$ is the Beta function defined by~\eqref{def:beta}. We get the
right hand side using relations~\eqref{eq:beta} and~\eqref{eq:gamma}.

The weight $f(S)=(h_1/h)^p$ satisfies $f(rS)=f(S)$ and, if $S_1\in \cH_1^{\epsilon}(0^q)$,
$f(S_1)=(h_1w)^p$, so

\begin{multline*}
\int\limits_{C(\cH_1^{\epsilon}(0^q))}\left(\frac{h_1}{h}\right)^p\dd\vol (T)=\\
=\int\limits_0^1 r^{2(q+1)-1}\dd r\int\limits_{\cH_1^{\epsilon}(0^q)}f(rS_1)\dd \vol(S_1)
=\frac{1}{2(q+1)}\int\limits_{\cH_1^{\epsilon}(0^q)}(h_1w)^p\dd \vol(S_1),
\end{multline*}
which completes the proof of the lemma.
\end{proof}

To continue the proof of Theorem~\ref{th:area-p}, we evaluate equation~\eqref{initial-integral}.
We
use lemma~\ref{lem:cusp-p} with 
$\epsilon'=\cfrac{\epsilon\cdot\sqrt{s^2+t^2}}{t}$ and integrate the function
$\left(\cfrac{t^2}{s^2+t^2}\right)^{p}$ that is the remaining part of 
$N_{\area_1^p}(S,\cC,\epsilon)$. So equation~\eqref{initial-integral} becomes

\[
\int_{C(K^{\epsilon})}N_{\area_1^p}(S,\cC,\epsilon)  \dd\nu(S)
=\frac{M  \cdot\Vol(\cH_1(\alpha'))\cdot2\pi\epsilon^2}{(p+1)\cdot (p+2)\cdots (p+q-1)}J_p+ o(\epsilon^2),
\]
where 
\[J_p=\int_0^1 s^{2n-1} \dd s
\int_0^{\sqrt{1-s^2}} t^{2q+1}\cdot
\left(\frac{t^2}{s^2+t^2}\right)^{p-1}
\,\,\dd t\]

Using polar coordinates $s=r\cos\theta$, $t=r\sin\theta$  we get
\[J_p(\cC)=\frac{1}{2(n+q+1)}\int_0^{\frac{\pi}{2}}(\cos\theta)^{2n-1}(\sin\theta)^{2p+2q-1}\dd\theta.\]
Using $u=\cos^2\theta$ we get
\begin{equation}\label{eq:area-p}
J_p(\cC)=\frac{\B(n,q+p)}{4(n+q+1)}=\frac{1}{4(n+q+1)}\frac{(n-1)!}{(p+q)\cdots (p+q+n-1)}.
\end{equation}

Using the fact that for $S\in C(K^{\epsilon})$, $N_{\area^p}(S,\cC,\epsilon)=\left(\frac{t^2}{s^2+t^2}\right)^p$ 
is invariant under scaling of $S$ and reduces to the initial definition of $N_{\area^p}(S_1,\cC,\epsilon)$ if 
$S_1\in K^{\epsilon}$, we can show as in the proof of lemma~\ref{lem:cusp-p} that 
\[\int\limits_{K^{\epsilon}}N_{\area_1^p}(S_1,\cC,\epsilon)\dd \vol(S_1)=
2(n+q+1)\int\limits_{C(K^{\epsilon})}N_{\area_1^p}(S,\cC,\epsilon)  \dd\nu(S),
   \]
where $2(n+q+1)=\dim_\R\cH(\alpha)$.
And so
\begin{multline*}
\int\limits_{K^{\epsilon}}N_{\area_1^p}(S,\cC,\epsilon)  \dd \vol(S)=\\
=M\cdot\pi\epsilon^2\cdot\Vol(\cH_1(\alpha'))\cdot\frac{(n-1)!}{(p+1)\cdots (p+q+n-1)}+ 
o(\epsilon^2).
\end{multline*}
Using  equation~\eqref{Siegel--Veech} we conclude that
\[c_{\area_1^p}(K,\cC)=
M\cdot\frac{\Vol(\cH_1(\alpha'))}{\Vol(K)}\cdot\frac{(n-1)!}{(p+1)\cdots (p+q+n-1)}.\]
Note that if we define $N_{\area_i^p}(S,\cC,\epsilon)$ in the same way as 
$N_{\area_1^p}(S,\cC,\epsilon)$, except that we use the area of the $i$-th cylinder,
for $i=1,\cdots,q$, then we have
$N_{\area^p}(S,\cC,\epsilon)=\sum_i N_{\area_i^p}(S,\cC,\epsilon)$.
It follows that $c_{\area^p}(K,\cC)=qc_{\area_1^p}(K,\cC)$, which completes
 the proof of Theorem~\ref{th:area-p}.
\end{proof}

\begin{remark}
(a)
 We note that the  volume  of  the  stratum  $\cH_1(\alpha')$ of  disconnected surfaces was   computed  in~\cite{EMZ}, equation (12); writing $\cH(\alpha')=\Pi_{i=1}^p\cH(\alpha_i')$ and $n_i=\dim_\C\cH(\alpha_i')$, we have
 \[
\Vol(\cH_1(\alpha'))=
\cfrac{1}{2^{p-1}}\cdot
\cfrac{\prod_{i=1}^p (n_i-1)!}
{(n-1)!}\cdot
\prod_{i=1}^p\Vol(\cH_1(\alpha'_i)).
\]
(b) As an example we consider the moduli space of tori. To have a configuration we need to mark one regular point.
The only possible topological type $\cC$ of configuration is  a closed saddle connection based at this regular point.
 We get, using Lemma~\ref{lem:cusp-p} for $p=0$ and $q=1$,
\[
c_{\conf}(\cH_1(0),\cC)=
\lim_{\epsilon\to 0}\cfrac{1}{\pi\epsilon^2}\cdot\cfrac{\Vol(\cH_1^{\epsilon}(0))}{\Vol(\cH_1(0))}=
\cfrac{1}{\pi\epsilon^2}\cdot\cfrac{2\pi\epsilon^2}{\pi^2/3} = \cfrac{6}{\pi^2}=\cfrac{1}{\zeta(2)},
\]
which is  the  well-known  factor  for  the proportion of  coprime lattice points in $\Z\oplus\Z$.
\end{remark}

 \subsection{Mean area of the periodic region}\label{sec:mean-area-conf}
 Fix an admissible topological type of configuration $\cC$ for a connected component $K$ 
of a stratum $\cH_1(\alpha)$ that comes with $q\geq 1$ cylinders.
Recall that  $N_{\area^p,\conf}(S,\cC,L)$ denotes the 
$p$-th power of the
total area of the periodic region (union of the cylinders) on $S$ coming from a 
configuration of topological
type $\cC$ whose length is at most $L$.

\begin{proof}[Proof of theorem~\ref{th:area-p-conf}]
 The argument is a special case of the argument in section~\ref{sec:mean-area}.
By decomposing as before a surface $S$ in $C(K^{\epsilon})$ as $sS_1'$ and $tT_1$ we use 
the  weight 
\[N_{\area^p,\conf}(S,\cC,\epsilon)=\left(\frac{t^2}{s^2+t^2}\right)^p.\]
So by taking $p=0$ in lemma~\ref{lem:cusp-p},
\begin{equation}\label{lem:cusp}
\int_{\cH_1^{\epsilon}(0^q)}\dd\vol (T)=\frac{2\pi \epsilon^2}{(q-1)!}.
\end{equation}

We then have 
\[
\int_{C(K^{\epsilon})}N_{\area^p,\conf}(S,\cC,\epsilon)  \dd\nu(S)
=M  \cdot\Vol(\cH_1(\alpha'))\cdot\frac{2\pi\epsilon^2}{(q-1)!}J_p+ o(\epsilon^2),
\]
where $J_p$ satisfies relation~\eqref{eq:area-p}. We conclude as in section~\ref{sec:mean-area}
that
\[
\int\limits_{K^{\epsilon}}N_{\area^p,\conf}(S,\cC,\epsilon)  \dd \vol(S)
=\frac{M\cdot\pi\epsilon^2\cdot\Vol(\cH_1(\alpha'))\cdot (n-1)!}{(q-1)!\cdot(p+q)\cdots (p+q+n-1)}+ o(\epsilon^2).
\]
It suffices to apply relation~\eqref{Siegel--Veech}.
\end{proof}

\subsection{Configurations with periodic regions of large area.}\label{sec:allcylindersp}

Fix an admissible configuration $\cC$ in a connected component $K$ of a stratum $\cH_1(\alpha)$ that comes with $q\geq 1$ cylinders and let $x\in [0,1)$ be a real parameter.
Recall that  $N_{\conf,A\ge x}(S,\cC,\epsilon)$ denotes the number of configurations  on $S$ of type $\cC$ of 
length  smaller than  $\epsilon$ and such that the total area of the $q$ cylinders is at least $x$ (of the area one surface $S$).  
 
 \begin{proof}[Proof of Theorem~\ref{th:big-area}]
We use a modification of the argument from section~\ref{sec:mean-area}. Here we consider the subset of  $K^{\epsilon}$ consisting of surfaces $S$ whose area of the periodic part is at least $x$ (of
the area $1$ surface $S$).
Construct $S$ in the cone of this set using  $s S_1'$ and $tT_1$ for some scalars $s,t$ and $S_1'\in \cH_1(\alpha')$ and $T_1\in \cH_1^{\epsilon'}(0^q)$. We need to evaluate the integral in equation~\eqref{initial-integral}
with an additional condition : when scaling $S$ by $\frac{1}{s^2+t^2}$ we get a surface of area one that satisfies 
\[\area\left( \cfrac{t}{\sqrt{s^2+t^2}}T_1\right)=  \cfrac{t^2}{s^2+t^2}>x\qquad\iff
\qquad t\geq\sqrt{ \cfrac{x}{1-x}}\;s.\]
 We count a configuration that satisfies this additional constraint with weight $1$, so equation~\eqref{initial-integral} becomes

\begin{multline*}
\int\limits_{C(K^{\epsilon})} N_{\conf,A>x}(S,\cC,\epsilon) \dd\nu(S)\\
=M\vol(\cH_1(\alpha'))
\int\limits_0^1 s^{2n-1} \dd s
\int\limits_{\sqrt{\frac{x}{1-x}}s}^{\sqrt{1-s^2}} t^{2q+1}\dd t
\int\limits_{\cH_1^{\epsilon'}(0^q)}\dd\vol (T)+ o(\epsilon^2),
\end{multline*}
where $\epsilon'=\cfrac{\epsilon\cdot\sqrt{s^2+t^2}}{t}$.

Using relation~\eqref{lem:cusp}, the right hand side becomes 
\[\frac{M\cdot2\pi\epsilon^2}{(q-1)!}\cdot\Vol(\cH_1(\alpha'))\cdot I_x+ o(\epsilon^2)\] where
\[
I_x=\int_0^{\sqrt{1-x}}s^{2n-1}\int_{\sqrt{\frac{x}{1-x}}s}^{\sqrt{1-s^2}}t^{2q+1}
       \frac{s^2+t^2}{t^2}\dd t\dd s.
\]

The domain of integration for $s$ and $t$ is described in Figure \ref{fig:domaininte}.

\begin{figure}[!htb]
\begin{center}\includegraphics[scale=1]{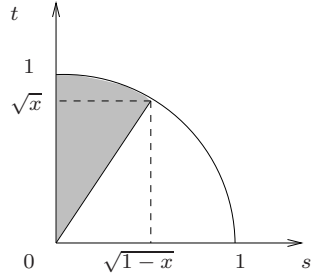}\end{center}
\caption{\label{fig:domaininte}Domain of integration.}
\end{figure}

Using polar coordinates $s=r\cos\theta$, $t=r\sin\theta$ and setting $\alpha=\arccos\sqrt{1-x}$ we get
\[I_x=\frac{1}{2(n+q+1)}\int_\alpha^{\frac{\pi}{2}}(\cos\theta)^{2n-1}(\sin\theta)^{2q-1}\dd\theta.\]
Using the change of variables $u=\cos^2\theta$ we get
\[
I_x=\frac{1}{4(n+q+1)}\int_0^{1-x}u^{n-1}(1-u)^{q-1}\dd u =\frac{B(1-x;n,q)}{4(n+q+1)},
\]
where $B(\cdot;\cdot,\cdot)$ is the incomplete Beta function as defined in equation~\eqref{def:incomplete-beta}.

Note that we have $c_{\conf}(K,\cC)=c_{\conf,A\ge 0}(K,\cC)$, so up to the same constant, the integral used to compute $c_{\conf,A\ge x}(K,\cC)$, resp $c_{\conf}(K,\cC)$ is given by $I_x$, resp $I_0$ (or equivalently $J_0$, see equation~\eqref{eq:area-p}), so \[\frac{c_{\conf,A\ge x}(K,\cC)}{c_{\conf}(K,\cC)}=\frac{I_x}{I_0}=\I(1-x;n,q).\]

We proved Theorem~\ref{th:big-area} of section~\ref{sec:results}, where, to expand $\I(1-x;n,q)$, we use Lemma~\ref{lem:combi4}. 
\end{proof}

\subsection{Configurations with a cylinder of large area}\label{sec:onecylindersp}

Fix an admissible configuration $\cC$ in a connected component $K$ of a stratum $\cH_1(\alpha)$ that comes with $q\geq 1$ cylinders and let $x$ be a parameter that satisfies $0\le x<1$.

We choose and fix one of the named closed saddle connection of $\cC$ that bounds a cylinder and call this cylinder the first one. 
Recall that $N_{\conf,A_1\ge x}(S,\cC,L)$ is the number of configurations of type  $\cC$ of length at most  $L$  and such that the area of the first cylinder is at least $x$ (of the area one surface $S$). 

\begin{proof}[Proof of theorem~\ref{th:Vorobets:constant}]
 Suppose that we have $q\geq 1$ cylinders. The argument is as in section~\ref{sec:allcylindersp} with the following modification: We replace ``$\area$'' by  ``area of the first cylinder'' in the condition  $\area \left( \frac{t}{\sqrt{s^2+t^2}}T_1\right)=  \frac{t^2}{s^2+t^2}>x$. 
Denote by $Cusp_a(\epsilon')$ the subset of $\cH_1^{\epsilon'}(0^q)$ of tori $T_1$ such that the area of the first cylinder of $T_1$ is at least $a=x\frac{s^2+t^2}{t^2}$. 
Using the usual parameters $(\gamma,h,h_i,t_j)$, a torus $T$ is in the cone of  $Cusp_a(\epsilon')$ if the area one torus $\frac{1}{\sqrt{wh}}T$ is in $Cusp_a(\epsilon')$, so the parameters for $T$ have the additional constraint $\frac{wh_1}{wh}>a$.

To compute the volume of $Cusp_a(\epsilon')$ we proceed as in the proof of lemma~\ref{lem:cusp-p}
for $x=0$.
The only difference is that we replace the condition $0\leq h_1\leq h$ by $ah\leq h_1\leq h$. So
$(h_1,\dots,h_{q-1})$ is in a cone whose volume is $ \frac{((1-a)h)^{q-1}}{(q-1)!}$. 
 So the volume is as in equation~\eqref{lem:cusp} except that we have an extra factor $(1-a)^{q-1}$. So we need to integrate 
\begin{equation}\label{int:Veech}
 I'_x=\int_0^{\sqrt{1-x}}s^{2n-1}\int_{\sqrt{\frac{x}{1-x}}s}^{\sqrt{1-s^2}}t^{2q+1}
         (1-a)^{q-1}\frac{s^2+t^2}{t^2}\;\dd t\dd s.
\end{equation}
Using polar coordinates $s=r\cos\theta$, $t=r\sin\theta$ followed by the change of variables $w=\frac{\cos^2\theta}{1-x}$  we get
\begin{equation}\label{eq:Veech}
I_x'=\frac{(1-x)^{n+q-1}}{4(n+q+1)}\;B(n,q).
\end{equation}
Note that $c_{\conf}(K,\cC)=c_{\conf,A_1\ge 0}(K,\cC)$, so $\frac{c_{\conf,A_1\ge x}(K,\cC)}{c_{\conf}(K,\cC)}=\frac{I_x'}{I_0'}$ as claimed.
(Recall that $\dim_\C \cH(\alpha)=n+q+1$.)
\end{proof}

\begin{remark} 
For $q=1$ we have  $c_{\conf,A_1\ge x}(K,\cC)=c_{\conf,A\ge x}(K,\cC)$ (see Theorem~\ref{th:big-area}).
\end{remark}

 \subsection{Correlation between the area of two cylinders.}\label{sec:correlation}

Let $\cC$ be an admissible  configuration for a connected component $K$ that comes with at least two cylinders.
Choose (and fix) two cylinders and let $x,x_1\in [0,1)$. Recall that we denote by $N_{A_2\geq x, A_1\geq x_1}(S,\cC,L)$ the number of configurations of length at least $L$ such that the area $A_1$ of the first cylinder is at least $x_1$ and such that the area $A_2$ of the second cylinder is at least $x(1-A_1)$.  We denote by $c_{A_2\geq x, A_1\geq x_1}(K,\cC)$ the corresponding Siegel--Veech constant. To simplify notation we will write $c_{A_1\geq x_1}(K,\cC)$ instead of $c_{\conf,A_1\geq x_1}(K,\cC)$.
  
\begin{proof}[Proof of theorem~\ref{th:correlation}]
 We proceed as in the proof of Theorem~\ref{th:Vorobets:constant} (see section~\ref{sec:onecylindersp}). Using the same notation, we have the condition that the area $\tilde A_1$ of the first cylinder of $\frac{t}{\sqrt{s^2+t^2}} T_1$ is at least $x_1$ and the area $\tilde A_2$ of the second cylinder of $\frac{t}{\sqrt{s^2+t^2}} T_1$ is at least $x(1-\tilde A_1)$. This means that the area $A_1$ of the first cylinder of $T_1$ satisfies  $A_1\geq a_1=x_1\frac{s^2+t^2}{t^2}$ and the area $A_2$ of the second cylinder is at least $x(1-\tilde A_1)\frac{s^2+t^2}{t^2}$ which gives $A_2\geq x \left(\frac{s^2+t^2}{t^2}- A_1\right)=a-xA_1$, where $a=x\frac{s^2+t^2}{t^2}$.
  
 Denote by $\Cusp_{a_1,a}(\epsilon')$ the subset of tori in $\cH_1^{\epsilon'}(0^q)$ that satisfies these conditions.
 Using the usual parameters $(\gamma,h,h_i,t_j)$, a torus $T$ is in the cone $C(\Cusp_{a_1,a}(\epsilon'))$ if and only if $T_1=\frac{1}{\sqrt{wh}}T$ is in $\Cusp_{a_1,a}(\epsilon')$. So the areas $A_1$ and $A_2$ of the first and second cylinders of $T_1$ must satisfy $A_1=\frac{wh_1}{wh}\geq a_1$ and $A_2=\frac{wh_2}{wh}\geq a- x\frac{wh_1}{wh}$.
 So we get
 \[
 a_1h\leq h_1\leq h, \quad (*)\qquad\qquad
 ah-xh_1\leq h_2\leq h-h_1. \quad (**)
 \]
 Equation $(**)$ has a solution if and only if $ah-xh_1\leq h-h_1$ which can be written as $h_1\leq \frac{1-a}{1-x}h$ so we need to modify $(*)$:
 \[
 a_1h\leq h_1\leq  \frac{1-a}{1-x}h.\quad (*')
 \]
Equation $(*')$ has a solution if and only if $a_1\leq \frac{1-a}{1-x}$. This translates into $x_2 \frac{s^2+t^2}{t^2}\leq 1$, where $x_2=x+x_1(1-x)$, which in turn becomes
 \[t\geq\sqrt{\frac{x_2}{1-x_2}} s.\]

 We also have $0\leq h_1\leq h$ and $0\leq h_2\leq h-h_1$. But  for all possible $t,s$ we have $a_1\geq 0$, $ah-xh_1\geq 0$, and $\frac{1-a}{1-x}\leq 1$, so $h_1$  can take all values between $a_1 h$ and $\frac{1-a}{1-x}h$ and $h_2$ can take all values between $ah-xh_1$  and $h-h_1$.

The computation of the volume of $\Cusp_{a_1,a}(\epsilon')$ is as in the proof of 
Lemma~\ref{lem:cusp-p} for $x=0$, except that we replace 
\[\int\limits_{\Delta^{q-1}(h)} \dd h_1\dots \dd h_{q-1}=\frac{h^{q-1}}{(q-1)!}\]
by
\begin{multline*}
\int\limits_{a_1h}^{\frac{1-a}{1-x}h}\;\dd h_1\int\limits_{ah-ph_1}^{h-h_1}\; \dd h_2\int\limits_{\Delta^{q-3}(h-h_1-h_2)} \dd h_3\dots \dd h_{q-1}\\
   =\frac{\left[(1-a)-(1-x)a_1\right]^{q-1}}{1-x} \frac{h^{q-1}}{(q-1)!}
   =\frac{(1-a_2)^{q-1}}{1-x} \frac{h^{q-1}}{(q-1)!},
\end{multline*}
where $a_2=x_2\frac{s^2+t^2}{t^2}$.

So the volume is as in equation~\eqref{lem:cusp} except that we have an extra factor $\frac{(1-a_2)^{q-1}}{1-x}$. So we need to integrate 
\[ \frac{1}{1-x} \int_0^{\sqrt{1-x_2}}s^{2n-1}\int_{\sqrt{\frac{x_2}{1-x_2}}s}^{\sqrt{1-s^2}}t^{2q+1}(1-a_2)^{q-1} \frac{s^2+t^2}{t^2}\;\dd t\dd s.\]
Note that up to the factor $1/(1-x)$ this is $I_{x_2}'$ as defined in~\eqref{int:Veech}.
The integral used to compute $c_{X_1\geq x_1}(K,\cC)=c_{\conf,X_1\geq x_1}(K,\cC)$ is $I_{x_1}'$, so we conclude
\[\frac{c_{A_2\geq x, A_1\geq x_1}(K,\cC)}{c_{A_1\geq x_1}(K,\cC)}=
\frac{I_{x_2}'}{1-x}\cdot\frac{1}{I_{x_1}'}.\]
Using equation~\eqref{eq:Veech} and $1-x_2=(1-x)(1-x_1)$ we find
\[\frac{c_{A_2\geq x, A_1\geq x_1}(K,\cC)}{c_{A_1\geq x_1}(K,\cC)}=
\frac{(1-x_2)^{n+q-1}}{(1-x)(1-x_1)^{n+q-1}}=(1-x)^{n+q-2}.\]
\end{proof}

\section{Extremal properties of configurations}

\subsection{Maximal total mean area of a configuration}\label{sec:maximal-abelian}

Consider an admissible topological type $\cC$ of configuration for some 
connected component $K$ of some stratum $\cH(\alpha)$, where $\alpha=(d_1,\dots,d_m)$ satisfies $d_i\geq 1$, for  $i=1,\dots,m$. We denote by $q(\cC)$ the number of cylinders that come with $\cC$.
In this section we prove Theorem~\ref{prop:qmax:dim:ab}, so we look for a topological type of configuration $\cC$ (that is admissible for some connected component $K$) that maximizes
\[c_{\MeanArea}(K,\cC)=
\frac{c_{\area}(K,\cC)}{c_{\conf}(K,\cC)}=
\frac{q(\cC)}{2g+m-2},\]
where for the last equality we used Corollary~\ref{cor:mean-area-p} and $c_{\cyl}(K,\cC)=q(\cC)c_{\conf}(K,\cC)$.

For a given stratum $\cH(\alpha)$, we start by determining the maximal possible number of cylinders $q_{\max}(\alpha)$ that can come from an admissible topological type of configuration for $\cH(\alpha)$ :
\[q_{\max}(\alpha)=\max_{\cC \text{ in } \cH(\alpha)} q(\cC).\]

 It is shown in~\cite{EMZ} that topological types of configurations can be constructed by creating  singular points of the following three types: 
 \begin{enumerate}
\item a cylinder, followed by $k\geq 1$ surfaces $S_i$ of genus $g_i\geq 1$ with figure eight boundary, followed by a cylinder.
See figure~\ref{fig:type I} for $k=3$ and $g_i=1$, $i=1,2,3$. We say that the newborn singularity is  of type $I$.
 \item a cylinder, followed by $k\geq 0$ surfaces $S_i$ of genus $g_i\geq 1$ with figure eight boundary, followed by a surface $S_{k+1}$ of genus $g_{k+1}\geq 1$ with a pair of holes boundary.
See Figure~\ref{fig:type II} for $k=2$ and $g_i=1$, $i=0,1,2$. We say that the newborn singularity is  of type  $II$.
 For $k=0$ we just have a cylinder followed by a surface with a pair of holes boundary.
 We might also reverse the order: a pair of holes torus followed by $k\geq 0$ figure eight tori followed by a cylinder.
  \item a singularity of type $III$, which is obtained from a singularity of type $II$ by  replacing the cylinder by a surface with a pair of holes boundary. As for surfaces of type $II$ 
  there might be  no surface with a figure eight boundary.
  See Figure~\ref{fig:type III} where all of the surfaces are of genus $1$ and where we have two surfaces with figure eight boundary.
 \end{enumerate}
\begin{figure}[!htb]
\begin{center}\includegraphics[scale=0.4]{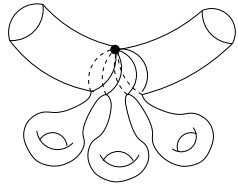}\end{center}
\caption{Block of surfaces creating a zero of type $I$}\label{fig:type I}
\end{figure}

\begin{figure}[!htb]
\begin{center}\includegraphics[scale=0.4]{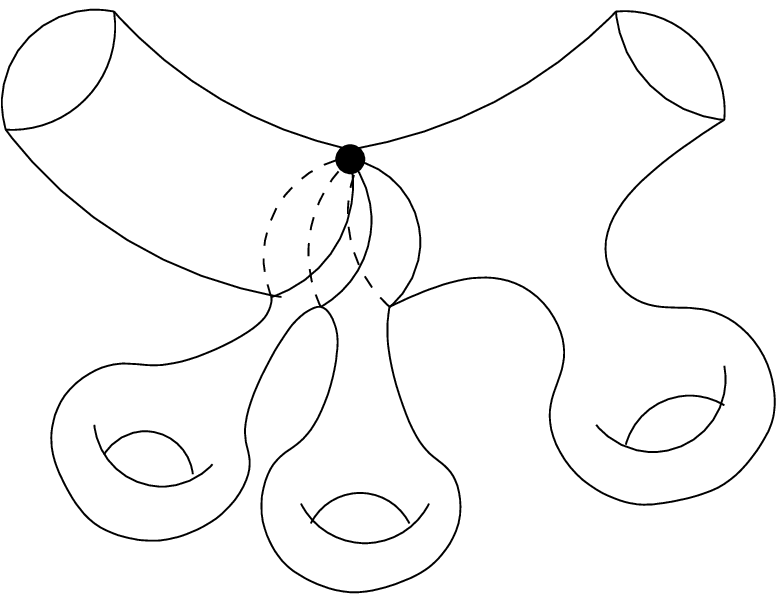}\end{center}
\caption{Block of surfaces  creating a zero of type $II$}\label{fig:type II}
\end{figure}

\begin{figure}[!htb]
\begin{center}\includegraphics[scale=0.4]{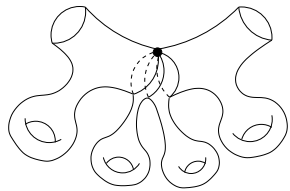}\end{center}
\caption{Block of surfaces  creating a zero of type $III$}\label{fig:type III}
\end{figure}

Counting angles, it is shown in~\cite{EMZ} that the order of the newborn zeros is as follows:
\begin{enumerate}
\item To create a zero of type $I$ one uses $k\geq 1$ figure eight boundaries that were created at zeros  of orders $a_1\geq 0$,\dots,$a_k\geq 0$ (a zero of order $0$ being a regular point). The zero then has order $\sum_{i=1}^k(a_i+2)$. All orders bigger or equal to $2$ are possible.
\item To create a zero of type $II$ one uses a figure eight boundary that was created at a zero of order $b'\geq 0$. If there is no figure eight boundary involved then the newborn zero has order $b'+1$.
  If there are $k\geq 1$ figure eight boundaries involved that were created at zeros of orders $a_1\geq 0$,\dots,$a_k\geq 0$ then the newborn zero has order $(b'+1)+\sum_{i=1}^k(a_i+2)$.
 All orders bigger or equal to $1$ are possible.
\item To create a zero of type $III$ we use two pair of holes boundaries created at zeros of orders $b_1'$ and $b_2''$. If there are no figure eight boundaries involved then the order of the newborn zero is $(b_1'+1)+(b_2''+1)$. If there are $k\geq 1$ figure eight boundaries involved that were created at zeros of orders $a_1\geq 0$,\dots,$a_k\geq 0$ then the newborn zero has order $(b'+1)+\sum_{i=1}^k(a_i+2)+(b_2''+1)$.
 All orders bigger or equal to $2$ are possible.
\end{enumerate}

For a given small $\gamma$ one can create all of the boundaries 
of the surfaces involved in the above construction by an appropriate figure eight surgery or a pair of holes surgery (see Figures~\ref{fig:figureeight} and~\ref{fig:twoholes}) such
that the boundaries all  have holonomy vector $\gamma$.

By arranging the blocks in a cyclic order and identify boundary components we create an admissible topological type of configuration of homologous saddle connections where each saddle connection is based at a newborn singularity of one of the three types. There is only the following obstruction: one needs to either use at least one surface with a pair of holes boundary or, if one only uses 
surfaces with figure eight boundaries, then one needs to use at least one cylinder.

Each cylinder is bounded by two saddle connections, so to have a one to one correspondence we think of the cylinder as being cut into two parts by the central waist curve. Each half cylinder then has a saddle connection $\gamma$ on the boundary (that is not the waist curve) that joins a saddle $P$ to itself. 
We say that $P$ accounts for this half-cylinder. A zero  of type $I$ (see Figure~\ref{fig:type I}) accounts for two half-cylinders (so for one cylinder), a zero of type $II$ (see Figure~\ref{fig:type II})
accounts for one half of a cylinder, and a zero of type
$III$ (see Figure~\ref{fig:type III}) does not account for any half cylinder.  
Note that only singularities of type $II$ can have order $1$.
It follows that if for $n\in\N$ we define 
\[
\chi(n)=
\begin{cases}
1/2,\ &\text{ if }n=1\\
1,\   &\text{ if }n>1
\end{cases}
\] then we have
$q_{\max}(d_1,\dots,d_k)\le \sum_{i=1}^m \chi(d_i)$. 
 To maximize the number of cylinders we construct the zeros  of orders greater than or equal to $2$ by zeros of type  $I$.
The other zeros of order $1$ must be created by zeros of type $II$ (coming from a cylinder followed by a torus with a pair of holes boundary). This works well if there is an even number 
of zeros of order $1$ in which case we constructed a configuration with $\sum_{i=1}^n \chi(d_i)$ cylinders. If there is an odd number of zeros of order $1$ then the construction of the zeros of order $1$ ends with a surface with a pair of holes boundary. In this case we are obliged to construct one of the zeros of order greater than $1$ by a surface of type $II$, so we constructed a configuration with 
$\sum_{i=1}^m \chi(d_i)-1/2$ cylinders. We showed
\begin{proposition}
\label{prop:bound:for:q:max}
Consider a stratum $\cH(\alpha)$, where $\alpha=(d_1,\dots,d_m)$ satisfies $d_i\geq 1$, for  $i=1,\dots,m$.
We have
\begin{equation}\label{eq:qmax}
q_{\max}(\alpha)= \left[\sum_{i=1}^n \chi(d_i)\right],
\end{equation}
where the square brackets denote the integer part of a number.
\end{proposition}

\begin{remark} The proposition says that it is  possible to find in each stratum $\cH(\alpha)$ a topological type of configuration $\cC$ such that $q(\cC)=q_{\max}(\alpha)$ is as stated.
It is not possible to find a topological type $\cC$ in each {\em connected component} of $\cH(\alpha)$ with $q_{\max}(\alpha)$ cylinders.
\end{remark}

We next show that given any connected component $K$ of a stratum $\cH(\alpha)$ and any admissible topological type $\cC$ of configuration for $K$, 
\[c_{\MeanArea}(K,\cC)\leq \frac{1}{3}.\]
\begin{proof}
Let $g\geq 2$. Denoting by $\ell(\alpha)$ the length of $\alpha$ we have 
\[c_{\MeanArea}(K,\cC)=
\frac{q(\cC)}{2g-2+\ell(\alpha)}\leq \cfrac{q_{\max}(\alpha)}{2g-2+\ell(\alpha)}.\]
We determine
\[
\max_{\alpha\in\Pi(2g-2)}\ \cfrac{q_{\max}(\alpha)}{2g-2+\ell(\alpha)},
\]
where $\Pi(2g-2)$ denotes the set of permutations of $2g-2$.

First let us prove that a partition $\alpha$ containing at least one entry $1$ is not maximizing. Indeed, if there is at least one {\em pair} of entries $1$ we can modify the initial partition $\alpha$ by replacing the elements $1,1$ with a single entry $2$. This does not change the value~\eqref{eq:qmax} of $q_{\max}(\alpha)$, but decreases the denominator in the ratio $q_{\max}/(2g-2+\ell(\alpha))$.

If there is a {\it single} entry $1$ in the partition $\alpha$, then $\sum_{i=1}^n \chi(d_i)$ is not an integer. We can modify $\alpha$ by deleting the entry $1$ and increasing some other entry by $1$. This operation does not change the value~\eqref{eq:qmax} of $q_{\max}(\alpha)$, but decreases the denominator in the ratio $q_{\max}/(2g-2+\ell(\alpha))$.

Thus we have proved that all of the entries of the partition maximizing the ratio $q_{\max}(\alpha)/(2g-2+\ell(\alpha))$ are greater then or equal to $2$. The formula~\eqref{eq:qmax} for such partitions simplifies to $q_{\max}(\alpha)= \ell(\alpha)$. Now note that for any $g\ge 2$ the function
\[
f_g(x)=\cfrac{x}{2g-2+x}=1-\cfrac{2g-2}{2g-2+x}
\] 
is strictly decreasing. Thus, among all partitions of $2g-2$ with entries strictly greater 
than one we have to chose the one maximizing $\ell(\alpha)$.
This is the partition $(2,\dots,2)$ where the order $2$ appears $g-1$ times. For this partition we get $q_{\max}(2,\dots,2)=g-1=\ell(\alpha)$  and so $c_{\MeanArea}(K,\cC)$ is bounded from above by $1/3$ as claimed.
\end{proof}

\begin{proof}[Proof of Theorem~\ref{prop:qmax:dim:ab}]
It suffices to show that for each genus $g$ there is an admissible topological type of configuration $\cC$ for a connected component $K$ of the stratum $\cH(2,\dots,2)$ ($g-1$ zeros of order $1$) such that $c_{\MeanArea}(K,\cC)$ attains the upper bound $1/3$.
 
 Note that for the  upper bound $1/3$ for $c_{\MeanArea}(K,\cC)$ we used $q_{\max}$.
As we explained in the proof of the relation~\eqref{eq:qmax}, the only way to obtain $q_{\max}(2,\dots,2)$ is to only use zeros of type $I$ (a cylinder followed by a torus with a figure eight boundary that was created at a regular point followed by a cylinder). Doing this we obtain a surface in $\cH(2,\dots,2)$ with a configuration that comes with $g-1$ cylinders. 

By Lemma 14.2 in~\cite{EMZ} the surface in $\cH(2,\dots,2)$ we constructed  that comes with a maximizing configuration has odd parity of spin structure, so this surface is in $\cH^{\odd}(2,\dots,2)$. (We recall in the next section the classification of connected components.)
 Proposition~\ref{prop:qmax:dim:ab} is proved.
 \end{proof}

\subsection{Configurations with simple complementary regions}\label{sec:simple-abelian} 

This section provides an answer to the following question of Alex Eskin and Alex Wright: is it possible to find in each connected component of a stratum an admissible topological type of configuration whose complementary regions are tori (with boundary) and cylinders. 
Motivations for this problem can be found in~\cite{W}.

The answer depends on the connected component, so we need to recall the  classification of connected components for  strata $\cH(\alpha)$ of Abelian differentials from~\cite{KZ}. Some connected components are characterized by the fact that they only contain hyperelliptic surfaces. For a surface $S$ in a stratum $\cH(d_1,\dots,d_n)$ where all $d_i$ are even one has the notion of {\em parity of spin structure} that is either $0$ or $1$.  We then have
\begin{theocite}[M.~Kontsevich, A. Zorich~\cite{KZ}]
Let $\cH(d_1,\dots,d_m)$ be a stratum of Abelian differentials on a surface of genus $g\geq 4$. 
The strata $\cH(2g-2)$ and $\cH(g-1,g-1)$ are the only strata to have a hyperelliptic component $\cH^{\hyp}(2g-2)$, resp. $\cH^{\hyp}(g-1,g-1)$.
Apart from these hyperelliptic components we have:
\begin{enumerate} 
\item If at least one of the $d_i$ is odd then there is only one non-hyperelliptic component. 
\item If all of the $\alpha_i$ are even then there are two non-hyperelliptic components,  $\cH^{\even}$ with even and $\cH^{\odd}$ with odd parity of spin structure.
\end{enumerate}
\end{theocite}

\begin{remark} The classification in~\cite{KZ} also covers $g=2,3$ but we will not need this.
\end{remark}

\begin{proposition}\label{prop:Alex}
 Let $\cH^{comp}(\alpha)$ denote a connected component of a stratum of Abelian differentials on a surface of genus $g\geq 5$.

Then: 
 \begin{enumerate}
 \item If $\cH^{comp}(\alpha)$ is  hyperelliptic (so $\cH^{\hyp}(2g-2)$ or, if $g-1$ is even, $\cH^{\hyp}(g-1,g-1)$) then it is not possible to find an admissible topological type of configuration whose complementary regions are only tori (with boundary) and cylinders.
 \item If $g$ is even and $\cH^{comp}(\alpha)=\cH^{even}(\alpha)$ then we can find a topological type of configuration  whose complementary regions are tori, cylinders and one surface of genus two. But it is not possible to only have tori and cylinders. 
\item In all remaining connected components one can explicitly construct an admissible topological type of configuration  whose complementary regions are tori and cylinders.
\end{enumerate}
\end{proposition}

\begin{proof}
First note that a configuration containing only tori and cylinders, or tori, cylinders, and at most one genus two surface does not occur in hyperelliptic strata, since a hyperelliptic surface can contain at most two closed homologous saddle connections, which rules out $g\geq 5$. Lemma $14.5$ in~\cite{EMZ} describes precisely configurations in hyperelliptic components of strata.

We suppose for what follows that $\cH^{comp}(\alpha)$ is not a hyperelliptic component.

Recall the description of singularities of types $I$, $II$, and $III$ from the previous section. If apart from cylinders, we only use surfaces of genus $1$, then we must do figure eight surgeries and pair of holes surgeries at regular marked points. So we have:
\begin{enumerate}
\item A zero of type $I$  has order $2k$, where $k\geq 1$ is the number of tori with figure eight boundaries (see Figure~\ref{fig:type I}).  
\item A zero of type $II$ has order $2k+1$, where  $k\geq 0$ is the number of tori with figure eight boundaries (see figure~\ref{fig:type II}).  
  \item A zero of type $III$ has order $2k+2$, where  $k\geq 0$ is the number of tori with figure eight boundaries (see figure~\ref{fig:type III}).
 \end{enumerate}
 Assume that $\alpha$ contains at least one odd $d_i$, so $\alpha=(2a_1,\cdots,2a_p,2b_1+1,\cdots, 2b_r+1)$ with $p\geq 0$ and $r\geq 1$ ($p=0$ corresponds to the case when all $\alpha_i$ are odd). 

We construct blocs of surfaces that contain zeros of type $II$ of orders  $2b_1+1$,\dots, $2b_r+1$.
 Note that since $\sum_i d_i$ is even, the number $r$ of odd $d_i$ is even, so in our construction the first and last surface is a cylinder. If there is at least one zero of even order then we construct in addition blocs of surfaces that contain zeros of type $I$ of orders $2a_1$,\dots,$2a_p$. For this construction the first and last surface is also a cylinder.  
Arranging the blocs in a cyclic order and identifying cylinders we get in each case an admissible topological type of configuration for $\cH(\alpha)$ whose complementary regions are cylinders and tori. 
 
 Suppose now that all $d_i$ in $\cH(\alpha)$ are even, so we can not have a zero of type $II$. We then either have only zeros of type $I$ or only zeros of type $III$ as having zeros of types $I$ and $III$ necessarily implies that we have a zero of type $II$.
 
It is easy to verify ($\S 14.1$ in~\cite{EMZ}) that if we only have zeros of type $I$ then the corresponding surface has an odd parity of spin structure and if we only only have zeros of type $III$ then the parity of the resulting surface is the parity of $g-1$. So for $g$ odd we are done, but for even $g$ the parity of the spin structure is $1$ in both cases. 
 
If we only have zeros of type $III$ but replace one of the pair of holes boundary tori with a genus $2$ surface with a pair of holes boundary then the parity of the spin structure of the resulting surface is the parity of $g$. So for even $g$ we constructed a surface of parity $0$. This completes the proof.
  \end{proof}
  
\section{Toolbox}\label{sec:toolbox}
We recall some well known facts about the Beta function and incomplete Beta function.

The (real) {\em Gamma function} is defined for each $t>0$ by
\[
\Gamma(t)=\int_0^\infty e^{-u}u^{t-1}\dd u.
\]
It satisfies 
\begin{equation}\label{eq:gamma}
\Gamma(t)=(t-1)\Gamma(t-1),\quad\text{so for $n\in\N$,}\quad\Gamma(n)=(n-1)!.
\end{equation}

The \emph{Beta function} defined for real numbers $a,b>0$ by
\begin{equation}\label{def:beta}
\B(a,b)=\int_0^1u^{a-1}(1-u)^{b-1} \dd u
\end{equation}
satisfies for positive real numbers $a,b$ and positive integers $n,m$ 
\begin{equation}\label{eq:beta}
\B(a,b)=\frac{\Gamma(a)\Gamma(b)}{\Gamma(a+b)}\qquad\qquad
\B(n,m)=\frac{(n-1)!(m-1)!}{(n+m-1)!}.
\end{equation}

The \emph{Incomplete Beta function} defined for real positive numbers $t,a,b$ by 
\begin{equation}\label{def:incomplete-beta}
\B(t;a,b)=\int_0^tu^{a-1}(1-u)^{b-1} \dd u
\end{equation}
satisfies
\[
\B(t;a,b)=\B(a,b)\sum_{k=a}^{a+b-1}{a+b-1\choose k}t^k(1-t)^{a+b-1-k}.
\]

The {\em regularized incomplete Beta function} is defined as
\[
 \I(t;a,b)=\frac{\B(t;a,b)}{\B(1;a,b)}=\frac{\B(t;a,b)}{\B(a,b)}.
 \]
 
 See figure~\ref{fig:graphesbis} for the density function $f(t)=\frac{d}{dt}\I(t;a,b)$ for various values for $a$ and $b$. See~\cite{D} for a historical development of the incomplete Beta function.

\begin{figure}[!htb]
\begin{center}\includegraphics[scale=0.5]{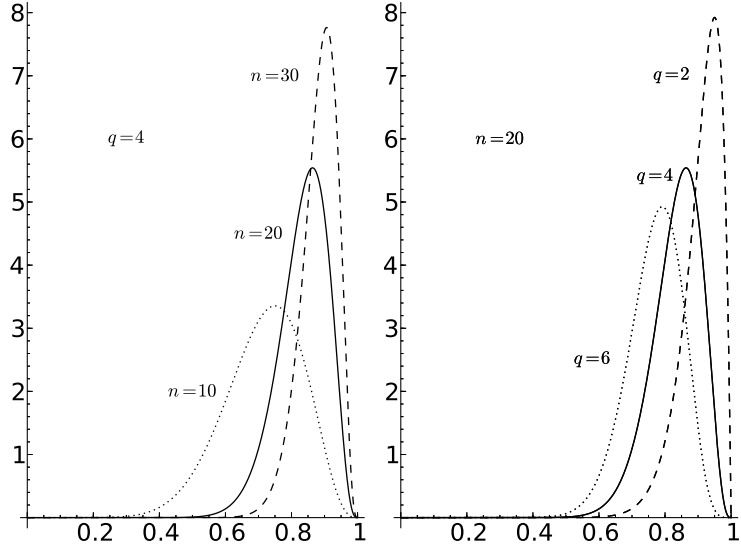}\end{center}
\caption{\label{fig:graphesbis} Graphs of the density function 
$f(t)=\frac{d}{dt}\I(t;a,b)$}
\end{figure}

\begin{lemma}\label{lem:combi1}
\begin{eqnarray}
I(A,B)=\sum_{k=0}^ B (-1)^ k{A+B\choose A+k}={A+B-1\choose B}\label{eq:tool1}\\
\tilde{I}(A,B)=\sum_{k=0}^{B}k(-1)^{k+1}{A+B\choose A+k}={A+B-2\choose B-1}\label{eq:tool2}
\end{eqnarray}
\end{lemma}

\begin{proof} Let $ I_x(A,B)=\sum_{k=0}^ B (-x)^ k{A+B\choose A+k}$\quad and\quad $\tilde{I}_x(A,B)=\sum_{k=0}^B k(-x)^k{A+B\choose A+k}$.

We have the recurrence relation $I_x(A,B+1)-(1-x)I_x(A,B)={A+B\choose A-1}$. Taking $x=1$ we get~\eqref{eq:tool1}.

Taking the derivative of the recurrence relation with respect to $x$ we get $\tilde{I}_x(A,B)=xI'_x(A,B)$ and hence $\tilde{I}_x(A,B+1)=-x{I}_x(A,B)+(1-x)\tilde{I}_x(A,B)$. Taking $x=1$ we get~\eqref{eq:tool2}.
\end{proof}

\begin{lemma}\label{lem:combi2} For $q\geq 0$ and $l\leq q$ the value of \[\hat{I}(n,q,l)=\sum_{k=0}^{q+1}{n+q+1\choose n+k} {k\choose q+1-l}(-1)^{l+k+q+1}\] is
\[\hat{I}(n,q,l)={n+l-1\choose l}.\]
\end{lemma}

\begin{proof} Using the recurrence relation on binomial coefficients one obtains $\hat{I}(n,q+1,l)=\hat{I}(n,q,l)$. So we have $\hat{I}(n,q+1,l)=\hat{I}(n,l,l)$. Noting that $\hat{I}(n,l,l)=\tilde I(n,l+1)$ we get the desired result.
\end{proof}

\begin{lemma}\label{lem:combi4} The incomplete Beta function satisfies
\[\B(1-x,n,q)=(1-x)^n\B(n,q)\sum_{l=0}^{q-1}{n+l-1\choose l}x^l.\]
\end{lemma}

\begin{proof}
\begin{align*}
&\sum_{k=n}^{n+q-1}{n+q-1\choose k}(1-x)^kx^{n+q-1-k}\\
=\quad& (1-x)^n\sum_{k=0}^{q-1}{n+q-1\choose n+k}(1-x)^{n+k}x^{q-1-k}\\
 =\quad&             (1-x)^n\sum_{k=0}^{q-1}
        \sum_{j=0}^k{n+q-1\choose n+k}{k\choose j}(-1)^{k-j}x^{q-1-j}\\
 =\quad&       (1-x)^n\sum_{l=0}^{q-1}
      \left[\sum_{k=0}^{q-1}{n+q-1\choose n+k}{k\choose q-1-l}(-1)^{k+l+q+1}\right]x^l
\end{align*}
It suffices to apply Lemma~\ref{lem:combi2}
\end{proof}

\begin{acknowledgement}
We would like to thank Anton Zorich for having initiated the present work. He asked most of the questions and initiated some answers. We thank Alex Wright for the formulation of the problem of section~\ref{sec:simple-abelian}, and for some comments on typos. We thank the anonymous referee(s) for the careful reading of the manuscript.   Both authors thank ANR ``GeoDyM'' for financial support.
\end{acknowledgement}

\bibliographystyle{amsalpha}

\end{document}